\documentclass[11pt]{amsart}
\usepackage{amsfonts}
\usepackage{amsmath}
\usepackage{amsthm}
\usepackage{bm,bbm}
\usepackage[dvips]{graphicx}
\usepackage{caption}
\usepackage{natbib}
\usepackage{color}
\usepackage[normalem]{ulem}
\usepackage{cancel}

\usepackage{bigints}

\usepackage{bm}

\numberwithin{equation}{section}

\allowdisplaybreaks

\theoremstyle{plain}
\newtheorem{theorem}{Theorem}[section]
\newtheorem{lemma}[theorem]{Lemma}
\newtheorem{proposition}[theorem]{Proposition}
\newtheorem{corollary}[theorem]{Corollary}
\theoremstyle{definition}
\newtheorem{definition}[theorem]{Definition}
\newtheorem{remark}[theorem]{Remark}
\newtheorem{example}[theorem]{Example}

\def\beqn{\begin{equation}}
\def\beqn*{$$}
\def\eeqn{\end{equation}}

\newcommand{\bj}{{\bf j}}
\def\P{\mathbb{P}}

\def\E{\mathbb{E}}

\newcommand{\reals}{{\mathbb R}}
\newcommand{\bbr}{\reals}

\newcommand{\bbn}{{\mathbb N}}
\newcommand{\bbq}{{\mathbb Q}}

\newcommand{\lk}{\text{lk}}

\newcommand{\bl}{{\bf l}}

\newcommand{\one}{{\mathbbm 1}}

\newcommand{\remove}[1]{}

\newcommand{\al}{\alpha}
\newcommand{\D}{\mathcal D}
\newcommand{\Var}{\text{\rm Var}}
\newcommand{\bk}{{\bf k}}

\newcommand{\beko}{\beta_{k+1}}
\newcommand{\etasig}{\eta_\sigma}
\newcommand{\etatau}{\eta_\tau}
\newcommand{\etasigjrko}{\eta_\sigma^{(j,r,k+1)}}

\newcommand{\G}{\mathcal G}

\newcommand{\W}{\mathcal W}
\newcommand{\Cov}{\text{\rm Cov}}

\newcommand{\ba}{{\bf a}}
\newcommand{\A}{\mathcal A}

\newcommand{\bp}{{\bf p}}
\newcommand{\balpha}{{\bm \alpha}}
\newcommand{\calL}{{\mathcal L}}

\begin{document}

\title[Limit theorems for topological invariants]
{Limit theorems for topological invariants of the dynamic multi-parameter simplicial complex}
\author{Takashi Owada}
\address{Department of Statistics\\
Purdue University \\
IN, 47907, USA}
\email{owada@purdue.edu}
\author{Gennady Samorodnitsky}
\address{School of Operations Research and Information Engineering\\
Cornell University \\
NY, 14853, USA}
\email{gs18@cornell.edu}
\author{Gugan Thoppe}
\address{Department of Computer Science and Automation \\
Indian Institute of Science \\
Bengaluru, India}
\email{gthoppe@iisc.ac.in}

\thanks{Owada's and Thoppe's research are partially supported by NSF grants, 
  DMS-1811428 and DMS 17-13012, respectively. Samorodnitsky's research is partially supported by the   ARO grant  W911NF-18 -10318 at Cornell  University.}


\begin{abstract}

\noindent 
Topological study of existing random simplicial complexes is non-trivial and has led to several seminal works. The applicability of such studies is, however, limited since the randomness in these models is usually governed by a single parameter. With this in mind, we focus here on the topology of the recently proposed multi-parameter random simplicial complex. In particular, we introduce a dynamic variant of this model and look at how its topology evolves.  In this dynamic setup, the temporal evolution of simplices is determined by stationary and possibly non-Markovian processes with a renewal structure.  Special cases of this setup include the dynamic versions of the clique complex and the Linial-Meshulum complex. Our key result concerns the regime where the face-count of a particular dimension dominates. We show that the Betti number corresponding to this dimension and the Euler characteristic satisfy a functional strong law of large numbers and a functional central limit theorem. Surprisingly, in the latter result, the limiting Gaussian process depends only upon the dynamics in the smallest non-trivial dimension. 
\end{abstract}

\maketitle

\section{Introduction}  \label{s:introduction}

The classical Erd\"os-R\'enyi graph $G(n, p)$ is a random graph on $n$ vertices in which each edge is present with probability $p$ independently. Even in such a simple model, answering topological questions such as  the threshold (in terms of the rate of decay of $p = p_n$ as $n\to\infty$) for connectivity (\cite{erdos:renyi:1959}) or for the existence of cycles (\cite{pittel:1988}) is completely non-trivial. Not surprisingly then, such a study becomes even more
interesting and difficult when posed in the context of random simplicial complexes---the higher dimensional generalizations of random graphs. Our focus in this work is on the general multi-parameter model of combinatorial random simplicial
complexes introduced by
\cite{costa:farber:2016,costa:farber:2017}. 


A summary of the recent progress made in the study of random complexes generalizing the Erd\"os-R\'enyi graph is as follows. The natural complex
built over any graph is its {\it clique complex}, otherwise known as
the {\it flag complex}, in which a set of vertices form a face or a
simplex if they form a clique in the
original graph. The topological properties of the random clique complex built over the Erd\"os-R\'enyi graph were
studied in \cite{kahle:2009}. This paper revealed, in particular,
the existence of a ``dominating dimension'', i.e.,~Betti numbers\footnote{The $k$th Betti number is a count of ``holes" of dimension $k + 1.$} of this dimension significantly exceed those of other dimensions, at least on average.

The $k$-dimensional
Linial-Meshulam complex is another important extension of the
Erd\"os-R\'enyi graph.  The $k = 2$ case of this model was  introduced by 
\cite{linial:meshulam:2006}, which was then extended to general $k$ by
\cite{meshulam:wallach:2009}. Here, 
one starts with a full
$(k-1)$-skeleton on $n$ vertices and 
then adds $k$-simplices with probability $p$
independently. Recently, topological features of the $k$-dimensional Linial-Meshulam complex,
with potential $k$-simplices weighted by independent standard uniform
random variables, were investigated by \cite{hiraoka:shirai:2017}, \cite{hino:2019},
\cite{Skraba:2020}, and \cite{Fraiman:2020}.

The  multi-parameter model introduced in
\cite{costa:farber:2016,costa:farber:2017} is a generalization of all of these models (see the next section for the formal definition). It was analyzed to some extent in
\cite{fowler:2019}, in which it was shown that a dominating dimension
exists in this model as well. In this work, we go beyond and examine the topological behavior in this dominating dimension as well as study its deviation from the expected behavior.

\cite{kahle:meckes:2013} did such a study in the context of random clique complexes
and proved a central limit
theorem for the dominating Betti number. To obtain an even deeper understanding, \cite{thoppe:yogeshwaran:adler:2016}  investigated the topological fluctuations in the dynamic variant of this model. Specifically, they considered the setup in which
every edge can change its state between being ON and being OFF,
i.e., between  being present and being absent, at the
transition times of a continuous-time Markov chain. They then derived a
functional central limit theorem for the Euler characteristic and the dominating 
Betti number of the resulting dynamic
clique complex. 

Within the context of the combinatorial simplicial complexes,
few attempts have been made at deriving
``functional-level" limit theorems for topological invariants (with a
few exceptions such as \cite{thoppe:yogeshwaran:adler:2016}, \cite{Skraba:2020}, and \cite{Fraiman:2020}). Our work fills in this gap. We introduce a dynamic variant of
the general multi-parameter random simplicial complex and derive a functional strong law of large numbers
and a functional central limit theorem for the Euler
characteristic and the dominating Betti number. Both of our results are proved in the space $D[0,\infty)$ of right continuous functions with left limits. Additionally, unlike \cite{thoppe:yogeshwaran:adler:2016}, we  do not assume a Markovian structure for the process according to which the
faces of the complex are switched on or off. Instead, the evolution here is determined by a stationary process with a renewal structure.  
Surprisingly, 
our key results indicate that the limiting Gaussian process in the
central limit theorem depends only upon the dynamics of the faces in
the smallest non-trivial dimension, irrespective of the dominating
dimension. This happens mainly because the faces in the smallest
non-trivial dimension are crucial for the existence of all higher
order faces.

The generality of our multi-parameter setup forces us to devise new
tools not needed under the random clique complex assumptions of
\cite{kahle:meckes:2013} and \cite{thoppe:yogeshwaran:adler:2016}. In
the latter case, for example, all Betti numbers of order greater than
the dominating dimension vanish with high probability. This is,
generally, not the case under our general setup. We solve this
difficulty by devising new ways of a much more detailed analysis of
these Betti numbers; see Section \ref{s:proofs.betti}. New coupling
arguments play a crucial role as well, especially in the proof of
functional strong laws  of large numbers. Such coupling arguments
enable one to stochastically dominate the face-counts in the dynamic
complex by those of a suitably defined static complex, e.g., see
\eqref{e:coupling.trick}.  We believe that such arguments could have
applications beyond the present context. 

This paper is organized as follows. In Section
\ref{s:dynamic.multi-para.complex}, we construct the dynamic
multi-parameter simplicial  complex and study some of its elementary
properties. A functional central limit theorem for the face counts in
this complex is 
stated in Section \ref{sec:face.counts}. Section
\ref{sec:top.invariants} contains the   main theorems for the Euler
characteristic and the Betti number in the dominating dimension. The
limit theorem for the face counts is proved in Section
\ref{s:proofs.facecounts}, and the limit theorems for the Euler
characteristic are proved in Section \ref{s:proofs.euler}, while the
limit theorems for the Betti numbers in the critical (dominating
dimension) are proved in Section \ref{s:proofs.betti}. Some of the
proofs are postponed to the Appendix. 

The following notation will be used throughout the paper. 
The cardinality of a set $A$ will be denoted by $|A|$. 
The indicator
function of an event will be denoted by $\one \{ \cdot \}$. For two
positive sequences $(a_n)$ and $(b_n)$ the notation $a_n \sim b_n$
means that $a_n/b_n\to 1$ as $n\to\infty$. The ``fat arrow''
$\Rightarrow$ is reserved for weak convergence, where the topology is
obvious from the context (in this paper it is mostly the Skorohod
$J_1$-topology on $D[0,\infty)$). The stochastic domination of a
random variable $X$ by a random variable $Y$ (meaning that $P(X\leq
x)\geq P(Y\leq x)$ for all $x$) is denoted by $X \stackrel{st}{\le}
Y$.

\section{The dynamic multi-parameter simplicial
  complex}  \label{s:dynamic.multi-para.complex}

We begin by recalling the original multi-parameter simplicial complex
introduced by
\cite{costa:farber:2016,costa:farber:2017}. Starting with the alphabet
$[n]=\{1,\ldots, n\}$ and parameters $\bp=\bp(n)=(p_1,\ldots,p_{n-1})$ with 
$p_i\in [0,1], \, i=1,\ldots,
n-1,$ one constructs the complex $X([n], \bp)$
incrementally, one dimension at a time. Specifically, begin with $X([n],
\bp)^{(0)}=[n]$. For $i=1,\ldots, n-1$, once the skeleton\footnote{The $i$th skeleton of a complex consists of all of its faces with dimension less than or equal to $i.$}
$X([n], \bp)^{(i-1)}$ has been constructed, add to $X([n], \bp)$ each
$i$-simplex\footnote{a subset of $[n]$ with cardinality $i + 1.$} whose boundary is in $X([n], \bp)^{(i-1)}$, with probability
$p_i$ independently of all other potential $i$-simplices. 
Note that the probabilities in $\bp$ may 
depend on $n$. 

Next, we  define the ``dynamic" version of the multi-parameter
simplicial complex with a parameter sequence $\bp$. 
The key ingredient for our construction is a
collection of independent stochastic processes  
\begin{equation}  \label{e:on.off.proc}
\big(  \Delta_{i,A}(t), \,  t\ge 0 \big), \ \ 1\le i \le n-1,  \ A\in \W_i,
\end{equation}
where $\W_i := \big\{ A \subseteq [n]: |A|=i+1 \big\}$. Each of the
processes in \eqref{e:on.off.proc} is a $\{ 0,1 \}$-valued
stationary  process and, for $1 \le i \le n-1$ and  $A\in \W_i$,  
\begin{equation} \label{e:dynamic.rule}
A \text{ forms an } i\text{-face at time } t \ \ \Leftrightarrow \ \  \Delta_{\ell, B}(t)=1 \ \ \text{for all } \ell\in\{ 1,\dots,i \}, \, B\in \W_\ell \, \text{ with } B\subseteq A. 
\end{equation}
Equivalently, $A$ does not form an $i$-face at time $t$  if and only if $\Delta_{\ell, B}(t)=0$ for some $\ell \in \{ 1,\dots,i \}$ and $B\in \W_\ell$ with $B\subseteq A$. 
We say that the process $\Delta_{i,A}$ is ``on" at time $t$ if
$\Delta_{i,A}(t)=1$, and it is ``off" otherwise. We assume that, 
 for each $i \ge 1$, $(\Delta_{i,A}, \, A\in \W_i)$ constitutes a
 family of (independent) processes with a common distribution. We often
 drop the subscript $A$ when only the dimension $i$ matters. 

To give a clear picture of our model, we provide a simple example for $n = 4$ in Figure~\ref{fig:01.proc}. 
In this case, there appears a $3$-face on $[4]=\{ 1,2,3,4\}$ if and only if the eleven independent processes $(\Delta_{i,A}, \, 1\le i \le 3, A\in \W_i)$ are all at an ``on" state. For example, such a $3$-face is present at time $t_0$. At time $t_1$, the process $\Delta_{1, \{1,3\}}$ is ``off", while all the others are ``on." Then, 
the $2$-faces $[1,2,3]$, $[1,3,4]$ and the $3$-face $[1,2,3,4]$ do not appear in the model, whereas all 
the other $2$-faces do exist.

\begin{figure}
\includegraphics[scale=0.4]{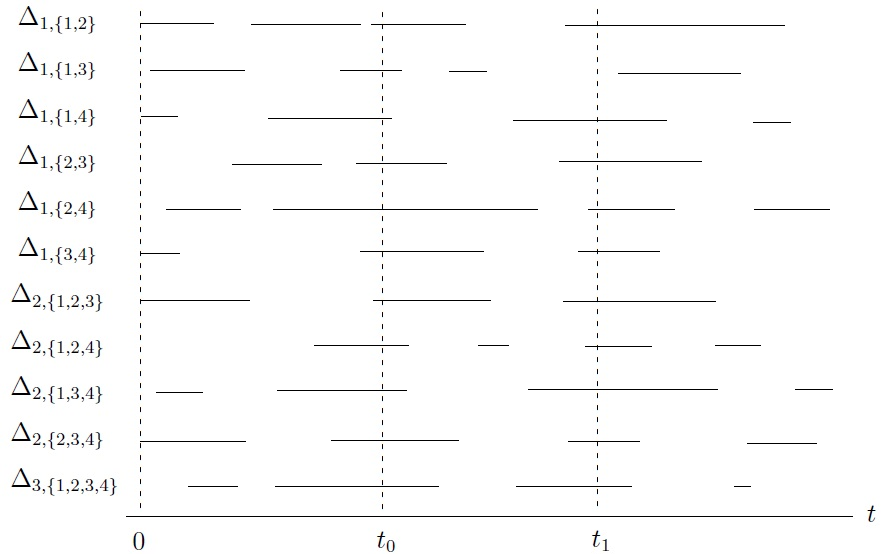}
\caption{\label{fig:01.proc} \footnotesize{Eleven independent stochastic processes with $n=4$. Each process stays at an ``on" state whenever a line segment appears, and it is at an ``off" state if the line segment disappears. }}
\end{figure}

We now model each $\Delta_i,$ $i = 1, 
\ldots, n - 1,$ via a specific $
\{0, 1
\}$-valued stationary renewal process. Let $\big( Z_j^{(i)}, \, j \ge 2  \big)$ be a sequence of
iid positive random variables with a common distribution function
$G_i$ and a finite positive mean $\mu_i$. 
The following assumption on the distribution functions $(G_i)$ will be
a standing assumption throughout the paper: letting $q:=\min\{ i\ge 1:
p_i <1 \}$, assume that 
\begin{equation} \label{e:assumption.Gi}
\text{there is $a>0$ such that} \ \ G_i(a)\leq 1/2 \ \ \text{for each
  $i=q,q+1,\ldots$.}
\end{equation}

Separately, let $\big( I_j^{(i)},
\, j \ge 0 \big)$ be a sequence of iid Bernoulli variables with
parameter $p_i.$ Finally, let $D^{(i)}$ be an \textit{equilibrium
random variable} with the distribution 
\begin{equation}  \label{e:equilibrium}
\P(D^{(i)} \le x) = \frac{1}{\mu_i} \int_0^x \big( 1-G_i(y)  \big)dy
=: (G_i)_e(x), \ \ x \ge 0.  
\end{equation}
All the random objects $\big( Z_j^{(i)} \big)$, $\big( I_j^{(i)}
\big)$, and $D^{(i)}$ are independent.  We define a delayed renewal
sequence by $S_0^{(i)}=0$, $S_1^{(i)}=D^{(i)}$, and  
\begin{equation}  \label{e:renewal.seq}
S_j^{(i)} = D^{(i)} + \sum_{\ell=2}^j Z_\ell^{(i)}, \ \ j \ge 2, 
\end{equation}
and the corresponding counting process, 
\begin{equation}  \label{e:stat.renewal.proc}
N_i(t) = \sum_{j=1}^\infty \one \{ S_j^{(i)} \le t \}, \ \ t \ge 0. 
\end{equation}
Since the first renewal time has the equilibrium distribution given
by \eqref{e:equilibrium}, the  delayed process $N_i$ in \eqref{e:stat.renewal.proc}
has stationary increments (\cite{ross:1996}). In particular, $\E\big( N_i(t) 
\big)=t/\mu_i$ . We finally define 
\begin{equation}
\label{e:Delta.process.rule}
\Delta_i(t) := \sum_{j=0}^\infty \one \big\{  S_j^{(i)} \le t <
S_{j+1}^{(i)} \big\} I_j^{(i)}, \ \  t \ge 0. 
\end{equation}

\begin{definition}
The dynamic multi-parameter simplicial complex $\big( X([n], \bp; t), \, t \ge 0\big)$ on $n$ vertices is defined by
\eqref{e:dynamic.rule}. For each dimension $i,$ the temporal evolution of
the $i$-dimensional faces  is determined by the independent processes $\big(
\Delta_{i,A}, \, 1\le i\le n-1, \, A\in \W_i \big)$ described in \eqref{e:Delta.process.rule}.
\end{definition}
\begin{remark}
As stated below in Lemma~\ref{l:cond.prob.Delta}, $\Delta_i$ is a stationary process for every $i$ which implies that $(X([n], \bp; t), t\geq 0)$ itself is stationary. In fact, for each $t \geq 0,$ $X([n], \bp; t)$ has the 
same distribution as that of the static multi-parameter simplicial
complex in \cite{costa:farber:2016,costa:farber:2017}.
\end{remark}
\begin{remark}
If  $\bp = (p, 1, 1, \ldots)$ and
$G_1(x)=1-e^{-\lambda x}$, $x\ge0$ for some $\lambda>0$, then $X([n],
\bp; t)$ is a reparametrization of the dynamic clique complex, for
which the evolution 
of the edges is determined by the $\{ 0,1 \}$-valued stationary continuous-time Markov chain (\cite{thoppe:yogeshwaran:adler:2016}).  
\end{remark}

The next result formally records the fact that, for each $i,$
$\Delta_i$ is a stationary process. It also states and proves a couple
of  useful properties concerning it. In particular, it shows that if
$p_i$ is small, then $\Delta_i$ is most of time off.
\begin{lemma}  \label{l:cond.prob.Delta}
(i) For every $i \in \{1, \dots,n-1\}$, $\big( \Delta_i(t), \, t \ge 0 \big)$ is a stationary process with $\P\big(\Delta_i(t)=1 \big)=p_i$. In addition, 
$$
\P\big( \Delta_i(t)=1\,  \big|\,  \Delta_i(0)=1\big) = 1 - (1-p_i)(G_i)_e(t), \ \  t \ge 0. 
$$

(ii) For every $i \ge q$ and $T>0$, 
\begin{equation} \label{e:sup.bound1}
\P\big( \sup_{0\le t \le T} \Delta_i(t)=1  \big) 
\leq p_i \left( 1+ (1-p_i)\frac{(G_i)_e(T)}{1-G_i(T)} \right). 
\end{equation}
\end{lemma}
\begin{proof}
The first statement in part (i) is obvious, because the process $N_i(t)$ has stationary increments. For the second one,  
\begin{align*}
\P\big( \Delta_i(t)=1\,  \big|\,  \Delta_i(0)=1\big) &=  \P \big(
    0 \le t  < D^{(i)} \bigr) + p_i P \big( 
    t  \geq  D^{(i)} \bigr) \\
&=1-(1-p_i)(G_i)_e(t).
\end{align*}
For Part $(ii)$, denote
$$
K= N_i(T) = \max\bigl\{ j\geq 1: S_j^{(i)}\leq T\bigr\} \ \ \text{($K=0$ if $S_1^{(i)}>T$).} 
$$
Then, 
\begin{align*}
\P\big( \sup_{0\le t \le T} \Delta_i(t)=1  \big) &= p_i + \P\big( \Delta_i(0)=0, \, \sup_{0 < t \le T} \Delta_i(t)=1  \big) \\  
&=p_i + (1-p_i)  \E\bigl[ 1-(1-p_i)^K\bigr].
\end{align*}
It is clear that $K$ is dominated by 
$$
K' := \begin{cases}
\min\{ j\ge 2: Z_j^{(i)}>T \}-1 & \text{if } D^{(i)} \le T \\
0 & \text{if } D^{(i)} >T. 
\end{cases}
$$
Evaluating the above
expression with $K$ replaced by $K'$ gives us
\eqref{e:sup.bound1}. 
\end{proof}

Sometimes we will also impose the following additional assumption on
the distributions $(G_i)$. 
\begin{equation}  \label{e:cond.regularity}
c := \sup_{i\ge q}\sup_{h>0,\,  0\leq y\leq 1
}\frac{G_i(y+h)-G_i(y)}{h^\gamma} < \infty \ \ \text{for some } 0 < 
\gamma \le 1, 
\end{equation}
Note that \eqref{e:cond.regularity} holds if $G_i$'s have a common
bounded density function (such as an exponential density). 

Under this additional assumption, we have the following estimates. 
\begin{lemma} \label{l:triples}
Assume  \eqref{e:cond.regularity}. Then for all
$0\leq r<s<t\leq 1$,
\begin{equation} \label{e:triple.1}
\P\bigl( \Delta_i(r)=0, \Delta_i(s)=1, \Delta_i(t)=0\bigr)
\leq \frac{2c}{a} p_i(t-r)^{1+\gamma}
\end{equation}
and
\begin{equation} \label{e:triple.2}
\P\bigl( \Delta_i(r)=1, \Delta_i(s)=0, \Delta_i(t)=1\bigr)
\leq \frac{2c}{a} p_i^2(t-r)^{1+\gamma}. 
\end{equation}
\end{lemma}
\begin{proof}
  Rewrite \eqref{e:triple.1} as
$$
p_i  \P\bigl( \Delta_i(r)=0,   \Delta_i(t)=0 \, \big|\, \Delta_i(s)=1
\bigr) \leq p_i \P\bigl( A_i(s) \leq s-r, R_i(s) \leq t-s\bigr),
$$
where $A_i$ and $R_i$ are respectively, the age and the residual lifetime
of a renewal process \eqref{e:stat.renewal.proc} with the interarrival distribution
$G_i$. It then follows from standard calculation in renewal theory (see e.g., \cite{resnick:1992}) that 
\begin{align*}
\P\bigl( A_i(s)\leq s-r, R_i(s)\leq t-s\bigr) &= \P(r \le S_{N_i(s)}^{(i)}, \, S_{N_i(s)+1}^{(i)} \le t)\\
&= \frac{1}{\mu_i} \int_0^{s-r} \big( G_i(y+t-s)-G_i(y)  \big)dy \leq \frac{2c}{a}(s-r)(t-s)^\gamma. 
\end{align*}
The last inequality comes from \eqref{e:assumption.Gi} and
\eqref{e:cond.regularity}. The argument for \eqref{e:triple.2} is
similar; since the process $\Delta_i$ is now required to be ``on" in
two distinct 
time intervals, $p_i$ in \eqref{e:triple.1} is replaced by $p_i^2$.
\end{proof}  

Recall that the probabilities in $\bp$ for the dynamic multi-parameter
simplicial complex $X([n],\bp; t)$ may depend on $n$. In the
sequel, following  \cite{costa:farber:2017}, we ``couple'' $\bp$ with
$n$ in a particular way: we set  $p_i =n^{-\alpha_i}$, $\alpha_i \in
[0,\infty]$  for $i=1,2,\ldots$. Accordingly,  we can work with an
infinite sequence $\balpha = (\alpha_1,\alpha_2,\dots)$, independent
of $n$, to control the rates at which the entries in $\bp$ decay. 
Below, we introduce some additional terms and notation, which we try to keep as consistent as possible with those in
\cite{costa:farber:2017}.

Let 
$$
\psi_j(\balpha) = \sum_{i=1}^j \binom{j}{i} \al_i, \ \ \ j\geq 1. 
$$
By convention, we set $\binom{j}{i}=0$ whenever $j < i$. Note that
$\psi_j(\balpha)$ is non-decreasing in $j$, i.e., $\psi_i (\balpha) \leq
\psi_j(\balpha)$ for each $\balpha$ and $i \le j$. We also let 
$$
\tau_j (\balpha) := j+1 - \sum_{i=1}^j \psi_i(\balpha) = j+1 - \sum_{i=1}^j\binom{j+1}{i+1}\al_i, \ \ 1 \le j \le n-1. 
$$
Additionally, we consider the following sets of parameters: 
$$
\D_j := \big\{ \balpha: \psi_j(\balpha) < 1 < \psi_{j+1}(\balpha)
\big\} 
$$
for $j\geq 1$ and $\D_0 := \{ \balpha:  \psi_1(\balpha)>1 \}$. 
 
Recalling the notation $q = q(\balpha)= \min \{ i \ge 1:\al_i>0 \}$ in
\eqref{e:assumption.Gi}, note that 
$$
\psi_j(\balpha) = 0, \ \ 
\tau_j(\balpha) = j+1, \ \ j=1,\dots,q-1. 
$$
Importantly, if $\balpha \in \D_k$ for some $k\ge q$, then 
$$
0< \psi_q(\balpha) < \dots < \psi_k(\balpha) < 1 < \psi_{k+1}(\balpha) < \dots,
$$
so that, 
$$
q = \tau_{q-1}(\balpha) <  \tau_{q}(\balpha) < \dots < \tau_k(\balpha) > \tau_{k+1}(\balpha) > \dots. 
$$
In this case, the index $k$ is referred to as \textit{the critical
  dimension}. Note that $\tau_j(\balpha)$, $j \ge k+1,$ can be
negative.  Observe also that, for $j >  k$,
\begin{align} \label{e:simple.lemma}
\tau_j(\balpha) -   ( \tau_{j+1}(\balpha) + \al_{j+1}) =& - 1 + \sum_{i=1}^j
  \bigg( \binom{j+2}{i+1}-\binom{j+1}{i+1} \bigg)\al_i  \\
=&   - 1 + \sum_{i=1}^j
  \binom{j+1}{i}\al_i >   - 1+\psi_j(\balpha) >0.    \notag 
\end{align} 

\section{Limit theorems for the face counts} \label{sec:face.counts} 

We consider the dynamic multi-parameter simplicial complex 
$\big( X([n], \bp; t), \, t \ge 0\big)$ constructed in the previous
section. Our basic assumption from now on will be that
\begin{equation} \label{e:basic.ass}
\balpha \in \D_k \ \text{for some $k\geq q$. }
\end{equation} 
Let  $\beta_{j,n}(t) := \beta_{j,n} \big( X([n], \bp; t) \big)$ 
be the $j$th (reduced) Betti number of the complex at time $t$. Note that $\bigl(
\beta_{j,n}(t), \, t\geq 0\bigr)$ is a stationary process. We will often use $\beta_{j, n}$ to mean 
$\beta_{j,n}(0)$. Similarly, we let $\chi_n(t)$
denote the Euler characteristic of the complex at time $t$. Then, 
$\bigl( \chi_{n}(t), \, t\geq 0\bigr)$ also is a stationary process, and $\chi_n$ will be used to denote $\chi_n := \chi_n(0)$. Recall that our goal is to establish functional strong
laws of large numbers (SLLN) and functional central limit theorems
(FCLT) for the Euler characteristic and the Betti number in the critical dimension $k$ of the dynamic multi-parameter simplicial
complex. This section is of preparatory nature and deals with the face counts
of the complex.

We write the face counts in dimension $j$ as  
\begin{align}  \label{e:k.face.counts}
f_{j,n}(t) &= \sum_{\sigma \subset [n], \, |\sigma| = j+1} \one \{
             \sigma \text{ forms a } j\text{-face in } X([n],\bp; t)
             \} =: \sum_{\sigma \subset [n],\,  |\sigma| = j+1}
             \xi_\sigma(t), \ \ t \ge 0. 
\end{align} 
Once again, let $\xi_\sigma := \xi_\sigma(0)$. As in \cite{kahle:meckes:2013} and \cite{thoppe:yogeshwaran:adler:2016},
we analyze the face counts first, and then relate them to the Euler
characteristic and the Betti numbers through the relations  
\begin{equation}  \label{e:EC.faces}
\chi_n (t) := \sum_{j=0}^{n-1} (-1)^j f_{j,n}(t), \ \  t\ge 0, 
\end{equation} 
and 
\begin{equation}  \label{e:Euler.characteristic}
\chi_n (t) := 1+ \sum_{j=0}^{n-1} (-1)^j\beta_{j,n}(t), \ \  t\ge 0. 
\end{equation}

We start with the asymptotic behaviour of the expected value and the covariances of the face counts. Note that not all results below
require the assumption \eqref{e:basic.ass}.
\begin{proposition} \label{p:moment.face.count}
For any $j\ge 1$, we have
$$
\E (f_{j,n}) \sim \frac{n^{\tau_j(\balpha)}}{(j+1)!}, \ \ n\to\infty. 
$$
Furthermore, for $j \ge q$ and $0 \le s \le t < \infty$, we have
\begin{align}
  \Cov \big( f_{j,n}(t), & f_{j,n}(s) \big)  \label{e:covfj} \\
& \sim \frac{n^{2\tau_j(\balpha)-\tau_q(\balpha)}}{(q+1)! \big( (j-q)!
  \big)^2}\, \big( 1-(G_q)_e(t-s) \big)   
   \vee\frac{n^{\tau_j(\balpha)}}{(j+1)!}\,  \prod_{i=q}^j \big(
 1-(1-p_i)(G_i)_e(t-s) \big)^{\binom{j+1}{i+1}} \notag 
\end{align}
as 
$ n\to\infty$, 
where $a\vee b = \max\{ a,b \}$ for $a, b \in \bbr$. 
In particular, if \eqref{e:basic.ass} holds, then 
\begin{equation}\label{e:varfk}
\Cov \big( f_{k,n}(t), f_{k,n}(s) \big) \sim \frac{n^{2\tau_k(\balpha)-\tau_q(\balpha)}}{(q+1)! \big( (k-q)! \big)^2}\, \big( 1-( G_q)_e(t-s) \big), \ \ n\to\infty. 
\end{equation}
\end{proposition}
\begin{remark}  \label{rem:non-random}
For $j<q$, $f_{j,n}(t)$ is, of course, nonrandom, 
so in this case, $\Cov\big(f_{j,n}(t), f_{j,n}(s) \big)=0$. 
\end{remark}
\begin{proof}
The asymptotics of the mean face count is easy to obtain. In fact, 
\begin{equation} \label{e:calculation.expectation}
\E(f_{j,n}) = \binom{n}{j+1} \prod_{i=q}^j p_i^{\binom{j+1}{i+1}} 
= \binom{n}{j+1} n^{\tau_j(\balpha) - (j+1)} \sim \frac{n^{\tau_j(\balpha)}}{(j+1)!} \ \ \text{as } n\to\infty. 
\end{equation}
For the covariances, we write
\begin{align*}
\E\big( f_{j,n}(t) f_{j,n}(s) \big) &= \sum_{\ell=0}^{j+1} \E \bigg( \sum_{\substack{\sigma \subset [n] \\ |\sigma|=j+1}} \sum_{\substack{\tau \subset [n] \\ |\tau|=j+1, \, |\sigma \cap \tau|=\ell}} \xi_{\sigma}(t)\xi_{\tau}(s)\bigg) \\
&=\sum_{\ell=0}^{j+1} \binom{n}{j+1} \binom{j+1}{\ell} \binom{n-j-1}{j+1-\ell} \E \big( \xi_\sigma(t)\xi_\tau(s) \big) \one \big\{ |\sigma \cap \tau|=\ell \big\}. 
\end{align*}
If $\ell\in \{ 0,1,\dots,q \}$, all faces of $\sigma\cap \tau$ exist
with probability one; thus, 
$$
\E \big( \xi_\sigma(t)\xi_\tau(s) \big)\, \one \big\{ |\sigma\cap\tau|=\ell \big\} = \bigg( \prod_{i=q}^j p_i^{\binom{j+1}{i+1}} \bigg)^2=n^{2\tau_j(\balpha)-2(j+1)}. 
$$
On the other hand, if $\ell\in\{  q+1, \dots, j+1\}$, we have 
\begin{align*}
\E \big( \xi_\sigma(t)\xi_\tau(s) \big) \one \big\{ |\sigma\cap\tau|=\ell \big\} &= \prod_{i=q}^j p_i^{\binom{j+1}{i+1}} \times \prod_{i=q}^j \P \big( \Delta_i(t)=1\, \big|\, \Delta_i(s)=1 \big)^{\binom{\ell}{i+1}} \times \prod_{i=q}^j p_i^{\binom{j+1}{i+1}-\binom{\ell}{i+1}} \notag\\
&=: A_n \times B_n \times C_n. 
\end{align*}
Here, $A_n$ is the probability of $\tau$ spanning a $j$-face at time
$s$, while $B_n$ is the conditional probability that all faces of
$\sigma\cap \tau$ are present  at time $t$, given that $\tau$ spans a
$j$-face at time $s$. Finally, $C_n$ is the conditional probability of $\sigma$
forming a $j$-face at time $t$, given that all faces of $\sigma \cap
\tau$ are present at time $t$.  
Calculating the product of three terms via Lemma \ref{l:cond.prob.Delta}, 
$$
A_n \times B_n \times C_n = n^{2\tau_j(\balpha)-\tau_{\ell-1}(\balpha)-2(j+1)+\ell} \prod_{i=q}^j \big( 1-(1-p_i)(G_i)_e(t-s) \big)^{\binom{\ell}{i+1}} . 
$$

By the stationarity of face counts, together with \eqref{e:calculation.expectation}, we have that
\begin{align*}
\E\big( f_{j,n}(t) \big) \E\big( f_{j,n}(s) \big) &=  \big( \E(f_{j,n}) \big)^2 =  \binom{n}{j+1}^2 n^{2\tau_j(\balpha) - 2(j+1)} \\
&= \sum_{\ell = 0}^{j+1}  \binom{n}{j+1} \binom{j+1}{\ell} \binom{n-j-1}{j+1-\ell} n^{2\tau_j(\balpha) - 2(j+1)}. 
\end{align*}
Combining all these results yields 
\begin{align*}
\Cov\big( f_{j,n}(t), f_{j,n}(s)\big) &= \sum_{\ell=q+1}^{j+1} \binom{n}{j+1}\binom{j+1}{\ell}\binom{n-j-1}{j+1-\ell} \\
&\quad \times n^{2\tau_j(\balpha)-\tau_{\ell-1}(\balpha)-2(j+1)+\ell} \Big\{  \prod_{i=q}^j \big( 1-(1-p_i)(G_i)_e(t-s) \big)^{\binom{\ell}{i+1}} - n^{\tau_{\ell-1}(\balpha)-\ell} \Big\} \\
&\sim \sum_{\ell=q+1}^{j+1}  \frac{n^{2\tau_j(\balpha)-\tau_{\ell-1}(\balpha)}}{\ell ! \big( (j+1-\ell)! \big)^2}\,  \prod_{i=q}^{\ell-1} \big(  1-(1-p_i)(G_i)_e(t-s) \big)^{\binom{\ell}{i+1}}\\
&\sim \frac{n^{2\tau_j(\balpha)-\tau_q(\balpha)}}{(q+1)! \big( (j-q)! \big)^2}\, \big( 1-(G_q)_e(t-s) \big) \\
&\qquad \qquad \qquad \qquad \vee\frac{n^{\tau_j(\balpha)}}{(j+1)!}\,  \prod_{i=q}^j \big( 1-(1-p_i)(G_i)_e(t-s) \big)^{\binom{j+1}{i+1}}, \ \ \ n\to\infty, 
\end{align*}
where the last equivalence comes from the fact that $\bigl(
\tau_\ell(\balpha), \, \ell\geq q\bigr)$ is a sequence that increases
for $\ell\leq k$ and then decreases.  For the derivation of \eqref{e:varfk}, use the fact that $2\tau_k(\balpha) - \tau_q(\balpha) \ge \tau_k(\balpha)$. 
\end{proof}

\begin{remark} \label{rk:critical}
It follows immediately from the proposition that, under the assumption
\eqref{e:basic.ass}, for every $j\not= k$, 
\begin{equation} \label{e:critical.dom}
\lim_{n\to\infty} \frac{\E (f_{j,n})}{\E (f_{k,n})} =
\lim_{n\to\infty} \frac{\Var (f_{j,n})}{\Var (f_{k,n})} =0.
\end{equation} 
That is, the face counts in the critical dimension dominate those in
the other dimensions both in their means and their variances. 
\end{remark}

The following corollary will be useful in the sequel. Since time parameter
plays no role due to stationarity, we remove it to simplify the notation.  Denote
\begin{equation} \label{e:M.alpha}
M(\balpha)= \min \big\{ i: \tau_i(\balpha)<0 \big\};
\end{equation}
this is a finite number since $\tau_i(\balpha) \to -\infty$ as $i\to\infty$. 
\begin{corollary} \label{cor:neg.tau}
As $n\to\infty$,
$$
\sum_{j=M(\balpha)}^\infty \E(f_{j,n}) \to 0\,.
$$
\end{corollary}
\begin{proof}
  It follows from \eqref{e:calculation.expectation} that 
$$
\E(f_{j,n}) \le n^{\tau_j(\balpha)}   \le \left(\frac{1}{n^\beta}\right)^{j+1},
$$
where 
$$
\beta  = \inf_{j \ge M(\balpha)} \bigl[-\tau_j(\balpha)/(j+1)\bigr].
$$
Note that $\beta>0$, since $\tau_j(\balpha) < 0$ for all $j \ge M(\balpha)$, and 
$$
\lim_{j\to\infty} \frac{-\tau_j(\balpha)}{j+1} = 
\lim_{j\to\infty} \Big\{  \sum_{i=1}^j\binom{j}{i} \frac{\al_i}{i+1}-1 \Big\} \ge \lim_{j\to\infty} \Big\{ \binom{j}{q} \frac{\al_q}{q+1} - 1 \Big\} = \infty. 
$$
Hence,
$$
\sum_{j=M(\balpha)}^\infty \E(f_{j,n}) \le \sum_{j=M(\balpha)}^\infty \left( \frac{1}{n^\beta} \right)^{j+1} \to 0, \ \ \ n\to\infty,
$$
as desired.
\end{proof}  

As stated below, the face counts in the critical dimension $k$ turn out to satisfy a
functional central limit theorem. The limit turns out to be a
stationary Gaussian process whose covariance function is given by the limit in   \eqref{e:varfk}. Specifically,
let $\big( Z_k(t), \, t\ge 0\big)$ be a zero-mean stationary Gaussian process
with covariance function  
\begin{equation}  \label{e:covariance.func}
R_k(t) = \E \big( Z_k(t)Z_k(0) \big) = 1-(G_q)_e(t), \ \ t \ge 0. 
\end{equation}
The basic sample path properties of this process are described in the 
next proposition. 
\begin{proposition}  \label{p:holder.continuous}
The process $Z_k$ admits a continuous version, whose sample paths are $\delta$-H\"older continuous~for any $\delta\in (0,1/2)$. 
\end{proposition}
\begin{proof}
Since  $Z_k$ is a stationary Gaussian process and 
$$
\E \big[ \big( Z_k(t)-Z_k(s) \big)^2  \big]  = 2(G_q)_e(|t-s|) \le
\frac{2}{\mu_q} |t-s|, 
$$
the claim follows from the Kolmogorov continuity criterion.  
\end{proof}

The statement below is a FCLT for the face counts in the critical
dimension $k$.  We view $f_{k,n}(\cdot)$ as a (piecewise constant)
random element of  
$D[0,\infty)$, the space of right continuous functions with left
limits, which is equipped with the Skorohod
$J_1$-topology. 
\begin{proposition}  \label{p:clt.face.counts}
Assume \eqref{e:basic.ass}.  Then, as $n\to\infty$,
\begin{equation}  \label{e:clt.face.counts}
\left( \frac{f_{k,n}(t)-\E(f_{k,n})}{\sqrt{\Var (f_{k,n})}}, \, t\geq 0\right)
\Rightarrow
\bigl( Z_k(t), \, t\geq 0\bigr)   
\end{equation}
in the sense of convergence of the finite-dimensional distributions.
If  the assumption  \eqref{e:cond.regularity}
is satisfied 
then \eqref{e:clt.face.counts} also holds in the sense
of weak convergence in the $J_1$-topology on $D[0,\infty)$.
\end{proposition}
The proof is deferred to Section \ref{s:proofs.facecounts}.

\begin{remark} \label{rk:whyq}
It is interesting and, initially, unexpected that only the state
change distribution $G_q$ in the lowest nontrivial dimension $q$
contributes to the asymptotics of the face counts in the critical
dimension. This is due to the fact that the ``flipping'' of a $q$-simplex 
from ``on'' to ``off'' or vice versa affects 
the distribution of $k$-simplices more 
than does any flipping in a different dimension.
Note that 
if $G_q$  is exponential  with mean $1/\lambda$, then
$R_k(t)=e^{-\lambda t}$ and $Z_k$ is the 
Ornstein-Uhlenbeck Gaussian process, as in
\cite{thoppe:yogeshwaran:adler:2016}.  

\end{remark}

\section{FCLT for topological invariants} \label{sec:top.invariants} 

In this section, we present the main results of this paper: the
functional SLLN and the FCLT for the Euler characteristic and the
Betti numbers in the critical dimension. We defer the
proofs to Sections \ref{s:proofs.euler} and \ref{s:proofs.betti}. 
 
We start with the strong laws of large numbers.
\begin{theorem}   \label{t:SLLN}
Assume \eqref{e:basic.ass}. Then,  as $n\to\infty$, 
\begin{equation}  \label{e:SLLN.Euler.char}
\left( \frac{\chi_n(t)}{n^{\tau_k(\balpha)}}, \, t\geq 0\right)\to \frac{(-1)^k}{(k+1)!} \ \
\text{a.s.}
\end{equation}
and
\begin{equation}  \label{e:SLLN.betti}
\left(\frac{\beta_{k,n}(t)}{n^{\tau_k(\balpha)}}, \, t\geq 0\right)\to
\frac{1}{(k+1)!} \ \ \text{a.s.} 
\end{equation}
in the $J_1$-topology on $D[0,\infty)$, where the right hand sides of
\eqref{e:SLLN.Euler.char} and \eqref{e:SLLN.betti} are viewed as
constant elements of $D[0,\infty)$. 
\end{theorem}

After stating the functional strong law of large numbers, we proceed,
as it is frequently done, with the functional central limit
theorem. Note the similarity with the corresponding limit theorem for
the face counts in Proposition \ref{p:clt.face.counts}.

\begin{theorem}  \label{t:clt.topological.invariants}
Assume \eqref{e:basic.ass}.  Then, as $n\to\infty$, 
\begin{equation}  \label{e:Euler.characteristic.func.conv}
\left( \frac{\chi_n(t) - \E(\chi_n)}{\sqrt{\Var (f_{k,n})}}, \, t\geq
  0\right)  \Rightarrow \bigl( Z_k(t), \, t\geq 0\bigr)
\end{equation}
and
\begin{equation}  \label{e:betti.func.conv}
\left( \frac{\beta_{k,n}(t) - \E(\beta_{k,n})}{\sqrt{\Var (f_{k,n})}}, \, t\geq 0\right) \Rightarrow
\bigl( Z_k(t), \, t\geq 0\bigr) 
\end{equation}
in the sense of convergence of the finite-dimensional distributions. 
  
In addition, assume \eqref{e:cond.regularity} and
\begin{equation}  \label{e:sharp.drop.k+1}
\tau_k(\balpha)-\frac{\tau_q(\balpha)}{2}>\tau_{k+1}(\balpha). 
\end{equation}
Then, \eqref{e:Euler.characteristic.func.conv} and \eqref{e:betti.func.conv} also hold in the sense of weak convergence in the $J_1$-topology on $D[0,\infty)$. 
\end{theorem}
\begin{remark} \label{rk:clt.explicit}
By Proposition \ref{p:moment.face.count},
\eqref{e:Euler.characteristic.func.conv} can be restated as 
$$
\left( \frac{\chi_n(t) - \E(\chi_n)}{n^{\tau_k(\balpha)-\tau_q(\balpha)/2}}, \, t\geq
  0\right)  \Rightarrow \bigl( \big\{ (q+1)!\big\}^{1/2} (k-q)! Z_k(t), \, t\geq 0\bigr).
$$
A similar reformulation is possible for \eqref{e:betti.func.conv}.
\end{remark}
\begin{remark} \label{rk:conjecture}
We think that \eqref{e:cond.regularity} alone is sufficient for weak
convergence in the $J_1$-topology on $D[0,\infty)$ in
\eqref{e:Euler.characteristic.func.conv} and 
\eqref{e:betti.func.conv}. We have chosen to assume
\eqref{e:sharp.drop.k+1} in order to simplify an already long and
technical argument. 
\end{remark}

\begin{example}  \label{ex:LM.clique}
The dynamic variants of the Linial-Meshulam complex and the clique complex are special cases of our model. An explicit form of Theorem~\ref{t:clt.topological.invariants} is stated here for these two setups.

The  Linial-Meshulam simplicial complex (see
\cite{linial:meshulam:2006,meshulam:wallach:2009}) 
corresponds, in our description, to $\balpha=(0,\dots,0,\al_k,\infty, \infty,\dots )$, with $0<\alpha_k<1$
in some position $k\geq 2$. This $k$ is then the critical dimension with $q=k$, and $\tau_k(\balpha)=k+1-\al_k$. Furthermore,
\eqref{e:basic.ass} holds. If $X([n],\bp; t)$ is the dynamic
Linial-Meshulam complex, then Theorem \ref{t:clt.topological.invariants}
says that 
\begin{equation} \label{e:LM}
\left( \frac{\chi_n(t) - \E(\chi_n)}{\sqrt{n^{k+1-\al_k}}} , \, t\geq
  0\right)  \Rightarrow \bigl( \big\{ (k+1)!\big\}^{1/2}  Z_k(t), \, t\geq 0\bigr),
\end{equation} 
at least in the sense of finite-dimensional distributions. 
  
Consider now the dynamic clique complex, for which 
 $\balpha=(\al_1,0, 0,\dots )$ with $0<\alpha_1<1$ and $\al_1 \neq 1/m$ for any $m\in \bbn$. Then,
 $q=1$ and the critical dimension is $k=\lfloor 1/\alpha_1\rfloor\geq
 q$.  Once again, \eqref{e:basic.ass} holds. Here,
 $\tau_k(\balpha)=k+1-\binom{k+1}{2}\al_1$ and 
 $\tau_q(\balpha)=2-\al_1$. Now, Theorem \ref{t:clt.topological.invariants}
says that  
\begin{equation} \label{e:CC}
\left( \frac{\chi_n(t) - \E(\chi_n)}{n^{k-\binom{k+1}{2}\al_1+\al_1/2}} , \, t\geq
  0\right)  \Rightarrow \bigl(  \sqrt{2}(k-1)!  Z_k(t), \, t\geq
0\bigr), 
 \end{equation}
 once again, at least  in the finite-dimensional distributions.

For both models, we also obtain corresponding results for the Betti numbers in the critical dimension. In the dynamic clique complex, if $G_1$ is an exponential distribution, then, as mentioned above,  $Z_k$ is a zero-mean stationary Ornstein-Uhlenbeck Gaussian
process, as in \cite{thoppe:yogeshwaran:adler:2016}. 

As for the technical conditions for tightness,
in the dynamic Linial-Meshulam complex, we only need to check \eqref{e:cond.regularity} just for $i=k$, while \eqref{e:sharp.drop.k+1} always holds as $\tau_{k+1}(\balpha)=-\infty$. In the case of a dynamic clique complex, one needs to check \eqref{e:cond.regularity}  just for $i=1$.  
On the other hand, \eqref{e:sharp.drop.k+1} reduces to $\al_1 > 4/(2k+3)$, implying that the corresponding functional convergence follows only when $4/5 < \al_1 < 1$ and the critical dimension is $k=\lfloor 1/\al_1\rfloor=1$. 

\end{example}
\begin{remark}
For the dynamic clique complex, the assumption \eqref{e:sharp.drop.k+1} fails in a certain range of the parameter. Therefore, Theorem \ref{t:clt.topological.invariants} does not claim the functional convergence in full generality, for the Euler characteristic and the Betti numbers in the critical dimension. On the other hand, \cite{thoppe:yogeshwaran:adler:2016} who only discuss this model, established tightness in full generality, and hence FCLT in the $J_1$-topology on $D[0,\infty)$.
The reason for this discrepancy is the generality of our setup. In particular, in
the dynamic clique complex, all Betti numbers except that in the
critical dimension are known to vanish with a very high probability 
(see \cite{kahle:meckes:2013}, \cite{kahle:2014}), which makes it possible to
obtain the required tightness in
\cite{thoppe:yogeshwaran:adler:2016}. In the general multi-parameter
simplicial complex, however, this is no longer necessarily the case,
and the Betti number in the dimension greater than the critical one
may not vanish; see Corollary~1.7 of \cite{fowler:2019}. To overcome the
resulting difficulty, we have imposed an extra condition
\eqref{e:sharp.drop.k+1}. We anticipate that the tightness holds
without that extra condition; one way to avoid this is via very
complicated fourth moment estimates for the Betti numbers based on the
expression in Proposition \ref{p:betti.representation}.

 
\end{remark}

\section{Proof of the FCLT for the face counts}  \label{s:proofs.facecounts}

In the sequel, we omit the subscript $n$ from all face count and Betti number notations. For example, we simply write $f_j(t)$, $\beta_j(t)$
etc. Everywhere, $C$ denotes a generic positive constant, which is
independent of $n$ but may vary between (or even within) the lines.  

We start with proving the finite-dimensional convergence in
Proposition \ref{p:clt.face.counts}. By the Cram\'er-Wold device, it
is enough to show that for all $0 \le t_1 < \dots <t_m <\infty$, $a_i
\in \bbr$, $i=1,\dots,m$, $m\ge1$,  
\begin{equation} \label{e:CW.faces}
\frac{\sum_{i=1}^ma_i \big( f_k(t_i)-\E(f_k) \big)}{\sqrt{\Var(f_k)}}
\Rightarrow\sum_{i=1}^m a_i Z_k(t_i) \ \ \text{in } \bbr.
\end{equation}
Clearly, it is enough to consider such choices of the coefficients
for which the variance in the right hand side of \eqref{e:CW.faces}
does not vanish, so fix such a set of coefficients.  

Let $J$ be the collection of $k-$faces or, equivalently, words of length $k+1$ in $[n].$ For $\bj \in J,$ let 
$$
X_\bj =\frac{\sum_{i=1}^m a_i \big( \xi_{\bj}(t_i) - \E
  (\xi_{\bj})\big)}{\sqrt{\Var \big( \sum_{i=1}^m a_i f_k(t_i)
    \big)}};
$$
recall that $\xi_{\bj}(t)$ is the indicator function that the $k$-face
associated with the word $\bj$ is ``on" at time $t$. Finally, define 
$$
W := \sum_{\bj \in J}X_{\bj} =\frac{\sum_{i=1}^m a_i \big( f_k(t_i)-
  \E (f_k)\big)}{\sqrt{\Var \big( \sum_{i=1}^m a_i f_k(t_i) \big)}}, 
$$
so that $\E(W)=0$ and $\Var (W)=1$.

In the terminology of
\cite{barbour:karonski:rucinski:1989} (see Equ.~(2.7) therein),  $\big\{ X_\bj, \,
\bj \in J \big\}$ constitutes a \textit{dissociated} set of random variables. To see  this, identify each $k-$face $\bj \in J$ by the tuple $\bj_{q} \equiv \Big(j_1, \ldots, j_{\binom{k + 1}{q + 1}}\Big),$ where each $j_i$ corresponds to a $q-$face in $\bj.$ For example, when $k = 3$ and $q = 1,$ identify the $3-$face $[1, 2, 3, 4]$ by the tuple $\big([1, 2], [1, 3], [1, 4], \ldots, [3, 4]\big).$ Then, for any sets $K, L \subset J_q := \{\bj_q: \bj \in J\}$
such that
\[
\left|\bigcup_{\bj_q
\in K} \left\{j_1, \ldots, j_{\binom{k + 1}{q + 1}}\right\} \cap \bigcup_{\bj_q' \in L} \left\{j'_1, \ldots, j'_{\binom{k + 1}{q + 1}}\right\}\right| = \emptyset, \] 
we have that
$(X_{\bj}: \, \bj_q \in K)$ is independent of $(X_{\bj}: \, \bj_q \in L)$. This verifies the claim that $\{X_\bj: \bj \in J\}$ is a dissociated set of random variables. 
We can thus invoke the central limit theorem of
\cite{barbour:karonski:rucinski:1989} for sums of dissociated random
variables.


The approach is to estimate the $L_1$-Wasserstein metric between the
distribution $\calL_W$ of $W$ and the standard normal distribution, i.e. 
$$
d_1(\calL_W,\calL_Y) = \sup_\phi \Big| \E \big( \phi(W) \big) -  \E \big( \phi(Y) \big) \Big|, 
$$
where $Y$ has the standard normal distribution and 
the supremum is taken over all  $\phi: \bbr \to
\bbr$ such that $\sup_{y_1\neq y_2}\big| \phi(y_1) - \phi(y_2)
\big|/|y_1-y_2|\le 1$.  Assuming we have shown that $d_1(\calL_W,\calL_Y)\to0$,
we have  $W\Rightarrow Y$ as $n\to\infty$. Furthermore, direct
applications of Proposition 
\ref{p:moment.face.count} and \eqref{e:covariance.func} yield 
$$
\frac{\Var \big( \sum_{i=1}^m a_i f_k(t_i) \big)}{\Var(f_k)} \to \Var \big( \sum_{i=1}^m a_i Z_k(t_i) \big), \ \ \ n\to\infty. 
$$
Therefore, $d_1(\calL_W,\calL_Y)\to0$ would give us 
$$
\frac{\sum_{i=1}^m a_i\big( f_k(t_i)-\E(f_k) \big)}{\sqrt{\Var(f_k)}}  \Rightarrow \Big\{ \Var \big( \sum_{i=1}^m a_i Z_k(t_i) \big)\Big\}^{1/2} Y \stackrel{d}{=} \sum_{i=1}^m a_i Z_k(t_i), \ \ \ n\to\infty,
$$
as required. 

It remains to actually show that $d_1(\calL_W,\calL_Y) \to 0$ as
$n\to\infty$. Let $L_{\bj} =\{ \bk \in J: |\bk \cap \bj| \geq q+1
\}$ 
be the dependency neighborhood of $\bj \in J$, that is, a collection of simplices
$\bk$  having at least one $q$-face in common with $\bj$. Then a
slight reformulation of (3.4) in 
\cite{barbour:karonski:rucinski:1989}  and
Proposition \ref{p:moment.face.count} shows that for a constant $C$
that may depend on the coefficients $a_1,\ldots, a_m$, but on nothing
else, 
\begin{align}
&d_1(\calL_W,\calL_Y) \le C \sum_{\bj \in J} \sum_{\bk \in L_\bj}\sum_{\bl \in L_\bj} \Big\{ \E \big(|X_{\bj}X_\bk X_\bl|\big) + \E \big(|X_\bj X_\bk|\big) \E\big(|X_\bl|\big)  \Big\} \notag \\
         &\quad \le \frac{C}{n^{3\tau_k(\balpha)-3\tau_q(\balpha)/2}}
       \sum_{i_1, i_2,    i_3=1}^m \sum_{\bj  \in J} \sum_{\bk\in
           L_\bj}\sum_{\bl\in L_\bj}   \bigg\{  \E\Big[ \big(
           \xi_{\bj}(t_{i_1}) + \E(\xi_{\bj}) \big) \big(
           \xi_{\bk}(t_{i_2}) + \E(\xi_{\bk}) \big) \big(
           \xi_{\bl}(t_{i_3}) + \E(\xi_{\bl}) \big)\Big]   \label{e:big.brace} \\
&\qquad \qquad\qquad \qquad\qquad \qquad \qquad \qquad \qquad \quad  + 2\E \Big[
                                                                                   \big( \xi_{\bj}(t_{i_1}) + \E(\xi_{\bj}) \big) \big( \xi_{\bk}(t_{i_2}) + \E(\xi_{\bk}) \big)   \Big] \E(\xi_{\bl}) \bigg\}.  \notag 
\end{align}
For fixed $\bj\in J,\bk\in L_\bj,\bl\in L_\bj$ denote 
$$
\ell_{12}= |\bj \cap \bk|, \ \ \ell_{13} =|\bj \cap  \bl|, \ \ \ell_{23} = | \bk \cap \bl|, \ \ \ell_{123}= |\bj \cap \bk \cap \bl|. 
$$
Since $\bk, \bl \in L_\bj$, it must be that $\ell_{12}\ge q+1$ and
$\ell_{13}\ge q+1$, whereas $\ell_{23}$ and $\ell_{123}$ can be less
than $q+1$.  
Given $\ell_{12}$, $\ell_{13}$, $\ell_{23}$, and $\ell_{123}$ as
above,  the expression between the braces in the right hand side of
\eqref{e:big.brace} can, up to a constant factor, be bounded by 
$$
\prod_{i=q}^k p_i^{3\binom{k+1}{i+1} - \binom{\ell_{12}}{i+1} -  \binom{\ell_{13}}{i+1} -  \binom{\ell_{23}}{i+1} +  \binom{\ell_{123}}{i+1}}
$$
For example, for $0 \le r \le s \le t<\infty$, by the
inclusion-exclusion formula, 
\begin{align*}
&  \E \big( \xi_{\bj}(r)\xi_{\bk}(s)\xi_{\bl}(t)  \big) \\
  =& \prod_{i=q}^k p_i^{3\binom{k+1}{i+1} - \binom{\ell_{12}}{i+1} -  \binom{\ell_{13}}{i+1} -  \binom{\ell_{23}}{i+1} +  \binom{\ell_{123}}{i+1}} \\
\times &\prod_{i=q}^k \P\big( \Delta_i(s)=1 \, \big|  \, \Delta_i(r)=1 \big)^{\binom{\ell_{12}}{i+1}-\binom{\ell_{123}}{i+1}} \prod_{i=q}^k \P\big( \Delta_i(t)=1 \, \big|  \, \Delta_i(s)=1 \big)^{\binom{\ell_{23}}{i+1}-\binom{\ell_{123}}{i+1}} \\
\times &\prod_{i=q}^k \P\big( \Delta_i(t)=1 \, \big|  \, \Delta_i(r)=1 \big)^{\binom{\ell_{13}}{i+1} - \binom{\ell_{123}}{i+1}}\prod_{i=q}^k \P\big( \Delta_i(s)=\Delta_i(t)=1 \, \big|  \, \Delta_i(r)=1 \big)^{\binom{\ell_{123}}{i+1}}
 \\
\le &\prod_{i=q}^k p_i^{3\binom{k+1}{i+1} - \binom{\ell_{12}}{i+1} -  \binom{\ell_{13}}{i+1} -  \binom{\ell_{23}}{i+1} +  \binom{\ell_{123}}{i+1}},
\end{align*}
and the terms of the other types can be bounded in a similar manner. 

Furthermore, observe that for every $\ell_{12}\ge q+1$, $\ell_{13} \ge
q+1$, $\ell_{23}\ge0$, and $\ell_{123} \ge 0$,  the number of the
corresponding terms in \eqref{e:big.brace}  does not exceed a constant
multiple of $ n^{3(k+1)-\ell_{12} - \ell_{13} - \ell_{23} +  \ell_{123}}$.  
Therefore, 
\begin{align*}
d_1(W,Y) &\le \frac{C}{n^{3\tau_k(\balpha)-3\tau_q(\balpha)/2}}\, \sum_{\ell_{12}=q+1}^{k+1} \sum_{\ell_{13}=q+1}^{k+1} \sum_{\ell_{23}=0}^{k+1} \sum_{\ell_{123}=0}^{\ell_{12} \wedge \ell_{13} \wedge \ell_{23}} \prod_{i=q}^k p_i^{3\binom{k+1}{i+1} - \binom{\ell_{12}}{i+1} -  \binom{\ell_{13}}{i+1} -  \binom{\ell_{23}}{i+1} +  \binom{\ell_{123}}{i+1}} \\
&\qquad \qquad \qquad \qquad \qquad \qquad \qquad \qquad   \qquad \qquad \qquad  \times n^{3(k+1)-\ell_{12} - \ell_{13} - \ell_{23} + \ell_{123}} \\
&=C \sum_{\ell_{12}=q+1}^{k+1} \sum_{\ell_{13}=q+1}^{k+1} \sum_{\ell_{23}=0}^{k+1} \sum_{\ell_{123}=0}^{\ell_{12} \wedge \ell_{13} \wedge \ell_{23}} n^{3\tau_q(\balpha)/2 - \tau_{\ell_{12}-1}(\balpha) - \tau_{\ell_{13}-1}(\balpha) - \tau_{\ell_{23}-1}(\balpha) + \tau_{\ell_{123}-1}(\balpha)}
\end{align*}
($a \wedge b = \min\{a,b \}$ for $a, b \in \bbr$). The latter sum is a
finite sum, and each term in it does not exceed  $C
n^{-\tau_q(\balpha)/2}$ which can be seen by noticing that
$\tau_{\ell_{23}-1}(\balpha) - \tau_{\ell_{123}-1}(\balpha) > 0$ and
setting $\ell_{12} = \ell_{13} = q+1$.  Therefore, the sum goes to $0$
as $n\to\infty$ and, hence, we have established the convergence of the
finite-dimensional distributions in Proposition
\ref{p:clt.face.counts}. 

In order to prove tightness in the $J_1$-topology, we use Theorem 13.5
in \cite{billingsley:1999}. By the stationarity of $f_k(t)$,
it is sufficient to show that for every $T>0$, there exists $B>0$ such
that  
$$
\frac{\E \Big[ \big( f_k(t)-f_k(s) \big)^2 \big( f_k(s)-f_k(r) \big)^2
  \Big]}{\big(\Var(f_k)\big)^2} \le B   (t-r)^{1+\gamma}  
$$
for all $0 \le r \le s \le t \le T$, $n\ge 1$, with $\gamma$ as in
\eqref{e:cond.regularity}. By Proposition \ref{p:moment.face.count},
we only need to show existence of $B$ such that 
\begin{equation}  \label{e:tightness.Thm13.5.face.counts}
\frac{\E \Big[ \big( f_k(t)-f_k(s) \big)^2 \big( f_k(s)-f_k(r) \big)^2
  \Big]}{n^{4\tau_k(\balpha)-2\tau_q(\balpha)}} \le B
(t-r)^{1+\gamma}. 
\end{equation}
This will be established while proving tightness in the proof of Theorem
\ref{t:clt.topological.invariants} below. \qed

\section{Proofs of the limit theorems for the Euler characteristic}  \label{s:proofs.euler}

We start with the strong law of large numbers. As in the last section, $C$ denotes  a generic positive constant, which is independent of $n$.
\begin{proof}[Proof of \eqref{e:SLLN.Euler.char} in Theorem \ref{t:SLLN}]
Fix $0< T < \infty$ for the duration of the proof. 
We first check that for each $j\ge 0$, 
\begin{equation}  \label{e:first.claim.SLLN}
\sup_{0 \le t \le T} \frac{\big| f_j(t)-\E(f_j)  \big|}{\E(f_k)} \to 0 \ \ \text{a.s.}
\end{equation}
If $j \in \{ 0,\dots, q-1 \}$, the left hand side is identically zero
(see Remark \ref{rem:non-random}). For $j \ge q$, by the
Borel-Cantelli lemma, it suffices to show that for every $\epsilon>0$, 
\begin{equation}  \label{e:Borel.Cantelli1}
\sum_{n=1}^\infty \P\Big( \sup_{0 \le t \le T} \big| f_j(t)-\E(f_j)
\big| > \epsilon \E(f_k) \Big)<\infty, 
\end{equation}
which will follow once we prove the following two statements: 
\begin{align}
\sum_{n=1}^\infty \P\Big( \sup_{0 \le t \le T}  f_j(t) > \E(f_j) + \epsilon\E(f_k) \Big)&<\infty, \ \ \text{and}  \label{e:Borel.Cantelli2}\\
\sum_{n=1}^\infty \P\Big( \inf_{0 \le t \le T}  f_j(t) < \E(f_j) - \epsilon\E(f_k) \Big)&<\infty. \label{e:Borel.Cantelli3}
\end{align}
Choose a positive integer $m$ so large that 
\begin{equation}  \label{e:restriction.epsilon}
\prod_{i=q}^j \left( 1+ \frac{ (G_i)_e(T/m)}{1-G_i(T/m)}  \right)^{\binom{j+1}{i+1}} < 1 + \frac{\epsilon}{2}. 
\end{equation}
By stationarity, 
$$
\P\Big( \sup_{0 \le t \le T}  f_j(t) > \E(f_j) + \epsilon\E(f_k) \Big) \le m \P \Big( \sup_{0 \le t \le T/m}  f_j(t) > \E(f_j) + \epsilon\E(f_k) \Big). 
$$
 
We now construct a new static multi-parameter simplicial
complex $X([n], \bp^{(1)})$ by setting $p_i^{(1)}=\P\big( \sup_{0\le t
  \le T/m}\Delta_i(t)=1 \big)$ for $i\geq 1$. If
$f_j^{(1)}$ is the $j$-face count in  this static complex, then, by a straightforward coupling argument,
\begin{equation}
    \label{e:coupling.trick}
    \sup_{0\le t \le T/m}f_j(t)  \stackrel{st}{\le} f_j^{(1)}. 
\end{equation}
\remove{To show \eqref{e:stochastically.dominated} note that 
$$
\sup_{0\le t \le T/m} f_j(t) \le \sum_{\sigma \subset [n], \, |\sigma|=j+1} \sup_{0\le t \le T/m} \xi_\sigma (t), 
$$
and, 
$$
\P \big( \sup_{0\le t \le T/m} \xi_\sigma(t)=1  \big) \le \P\big( \sigma \text{ forms a } j\text{-simplex in } X([n], p^{(1)}) \big)
$$
as required. }
Since by part (ii) of Lemma \ref{l:cond.prob.Delta} 
 and \eqref{e:restriction.epsilon}, 
\begin{align*}
\E(f_j^{(1)}) &= \binom{n}{j+1} \prod_{i=q}^j
                (p_i^{(1)})^{\binom{j+1}{i+1}} \\
                &\le \binom{n}{j+1}
                \prod_{i=q}^j p_i^{\binom{j+1}{i+1}} \prod_{i=q}^j
                \left( 1+ \frac{(G_i)_e(T/m)}{1-G_i(T/m)}  \right)^{\binom{j+1}{i+1}} \le
                \Big(1+\frac{\epsilon}{2}\Big)\E(f_j), 
\end{align*}
we conclude that 
\begin{align*}
\P \Big( \sup_{0 \le t \le T/m}  f_j(t) > \E(f_j) + \epsilon\E(f_k)
  \Big) &\le \P \big( f_j^{(1)}-\E(f_j^{(1)}) > \E(f_j) +\epsilon
          \E(f_k) -\E(f_j^{(1)})  \big)   \\
&\le \P\Big(  f_j^{(1)}-\E(f_j^{(1)}) > \epsilon\E(f_k)
                                               -\frac{\epsilon}{2}
                                               \E(f_j) \Big).  
\end{align*}
As $\E(f_j)/\E(f_k)\to0$, $n\to\infty$ for $j\neq k$, it holds that, for sufficiently large $n$, 
\begin{align*}
 \P\Big(  f_j^{(1)}-\E(f_j^{(1)}) > \epsilon\E(f_k) -\frac{\epsilon}{2} \E(f_j) \Big) &\le \P \Big(  \big| f_j^{(1)}-\E(f_j^{(1)})  \big| > \frac{\epsilon}{2}\E(f_k) \Big) \\
&\le \frac{4}{\epsilon^2}\, \frac{\Var(f_j^{(1)})}{\big( \E(f_k) \big)^2} \le C \frac{\Var(f_j^{(1)})}{n^{2\tau_k(\balpha)}}, 
\end{align*}
where the last inequality comes from Proposition
\ref{p:moment.face.count}. Further, since each $p_i^{(1)}$ is 
asymptotically bounded by $p_i$ times a positive constant for $i=q,\ldots, j$, the argument of the
above proposition shows that for large enough $n$,
$$
\Var(f_j^{(1)})\le C_j^{(1)} n^{2\tau_j (\balpha)-\tau_q(\balpha)}
\vee  C_j^{(2)} n^{\tau_j(\balpha)} 
$$
for some finite positive constants  $C_j^{(1)}$ and $C_j^{(2)}$. Hence,  
\begin{align*}
\P \Big( \sup_{0 \le t \le T/m}  f_j(t) > \E(f_j) + \epsilon\E(f_k) \Big) &\le C \frac{n^{2\tau_j (\balpha)-\tau_q(\balpha)}  \vee n^{\tau_j(\balpha)}  }{n^{2\tau_k(\balpha)}} \\
&\le Cn^{-\tau_q(\balpha)} \le Cn^{-\tau_1(\balpha)} = Cn^{-(2-\al_1)}. 
\end{align*}
As $\al_1 = \psi_1(\balpha) \le \psi_k(\balpha) < 1$, we get $\sum_{n=1}^\infty n^{-(2-\al_1)}<\infty$, and so \eqref{e:Borel.Cantelli2} holds. 

We now turn our attention to \eqref{e:Borel.Cantelli3}. The
stationarity of $f_j(t)$  implies that  
$$
\P \Big( \inf_{0 \le t \le T}f_j(t) < \E(f_j) - \epsilon \E(f_k)  \Big) \le m \P \Big( \inf_{0 \le t \le T/m}f_j(t) < \E(f_j) - \epsilon \E(f_k)  \Big),  
$$
where this time $m$ is chosen so that 
$$
\prod_{i=q}^j \big( 1-(G_i)_e(T/m) \big)^{\binom{j+1}{i+1}} > 1-\frac{\epsilon}{2}. 
$$
Once again, we construct a new static multi-parameter simplicial
complex $X([n], \bp^{(2)})$ by setting this time $p_i^{(2)}=\P\big( \inf_{0\le t
  \le T/m}\Delta_i(t)=1 \big)$ for $i\geq 1$. If
$f_j^{(2)}$ is the $j$-face count in this static complex, then,
$f_j^{(2)}\stackrel{st}{\le}\inf_{0\le t \le T/m}f_j(t)$. Notice that
for $i \ge q$, 
$$
p_i^{(2)} \ge \P\big( \Delta_i(0)=1, \, D^{(i)} \ge T/m  \big) = p_i
\big( 1-(G_i)_e (T/m)  \big), 
$$
so by the choice of $m$, 
$$
\E(f_j^{(2)}) \ge \Big(  1-\frac{\epsilon}{2} \Big) \E(f_j). 
$$
Proceeding as above we conclude that, for sufficiently large $n$, 
$$
\P \Big( \inf_{0 \le t \le T/m}f_j(t) < \E(f_j) - \epsilon \E(f_k)
\Big)  \le C\frac{\Var(f_j^{(2)})}{n^{2\tau_k(\balpha)}}.
$$
Noting that $p_{i}^{(2)} \leq p_i^{(1)}$, the same logic as above tells that 
$$
\Var(f_j^{(2)})\le  C_j^{(1)} n^{2\tau_j (\balpha)-\tau_q(\balpha)}
\vee  C_j^{(2)} n^{\tau_j(\balpha)} 
$$
for some finite positive constants  $C_j^{(1)}, C_j^{(2)}$, and
\eqref{e:Borel.Cantelli3} follows in the same way as
\eqref{e:Borel.Cantelli2} did. 

The next step is to show that as $n\to\infty$, 
\begin{equation}  \label{e:Euler.as.conv0}
\sup_{0\le t \le T} \frac{\big| \chi_n(t) - \E(\chi_n) \big|}{\E(f_k)} \to 0 \ \ \text{a.s.,}
\end{equation}
and by stationarity it is enough to prove that 
\begin{equation}  \label{e:Euler.as.conv}
\sup_{0\le t \le T/m} \frac{\big| \chi_n(t) - \E(\chi_n) \big|}{\E(f_k)} \to 0 \ \ \text{a.s.}
\end{equation}
for an integer $m$ large enough so that $T/m\leq a/4$; the constant $a$ is given in 
the assumption \eqref{e:assumption.Gi}. It is not difficult to see that the choice of $m$ implies $(G_i)_e(T/m) \leq 1/2.$ Combining this with part (ii) of Lemma
\ref{l:cond.prob.Delta} and recalling that
$p_i^{(1)}=\P\big( \sup_{0\le t \le T/m}\Delta_i(t)=1 \big)$, we get 
$p_i^{(1)}\leq p_i(2-p_i)$. It is now
elementary to check that there is a function $h:[0,\infty]\to
[0,\infty]$ with $h(0)=0$, $h(\infty)=\infty$, and $h(\alpha)\in
(0,\infty)$ for $0<\alpha<\infty$, such that 
\begin{equation}  \label{e:alpha.tilde.rate}
p_i^{(1)}\leq  p_i (2-p_i) \le  n^{-h(\alpha_i)} \ \ \text{if} \ \ p_i=n^{-\alpha_i}, \
i\geq 1;
\end{equation}
(for example, one may take $h(\al)=\al-\log(2-2^{-\al})/\log2$). Define now $\tilde\balpha$ by $\tilde\alpha_i=h(\alpha_i)$,
$i=1,2,\ldots$. Then, $M(\tilde\balpha)$ defined by \eqref{e:M.alpha} is finite, and 
we use \eqref{e:EC.faces} to bound 
$$
\sup_{0\le t \le T/m} \frac{\big| \chi_n(t) - \E(\chi_n)
  \big|}{\E(f_k)} \le \sum_{j=0}^{ M(\tilde\balpha)-1}
\sup_{0\le t \le T/m} \frac{\big| f_j(t) - \E(f_j) \big|}{\E(f_k)} +
\sum_{j= M(\tilde\balpha)}^{n-1} \sup_{0\le t \le T/m}
\frac{\big| f_j(t) - \E(f_j) \big|}{\E(f_k)}. 
$$
By \eqref{e:first.claim.SLLN}, the first sum in the right hand side
almost surely goes to $0$ as $n\to\infty$. For the second sum, we
again use the Borel-Cantelli lemma by initially showing that, for
every $\epsilon>0$,  
$$
\sum_{n=M(\tilde\balpha)+1}^\infty \P \Big( \sum_{j=M(\tilde\balpha) }^{n-1} \sup_{0 \le t \le T/m} \big| f_j(t)-\E(f_j)  \big| > \epsilon \E(f_k)  \Big) < \infty. 
$$
Using Markov's inequality and recalling our notation for the face counts in the static multi-parameter simplicial
complex $X([n], \bp^{(1)}),$ we bound the above sum by
\begin{equation}  \label{e:Markov.fj1}
\frac{2}{\epsilon}\sum_{n=M(\tilde\balpha)+1}^\infty\frac{1}{\E(f_k)}
\sum_{j=M(\tilde\balpha)}^{n-1} \E \Big( \sup_{0\le t \le T/m}f_j(t)  \Big) \le
\frac{2C}{\epsilon}\sum_{n=1}^\infty \frac{1}{n^{\tau_1(\balpha)}}
\sum_{j=M(\tilde\balpha)}^\infty \E \big( f_j^{(1)}  \big)<\infty 
\end{equation} 
since $\sum_{n=1}^\infty n^{-\tau_1(\balpha)}<\infty$ and
$\sum_{j= M(\tilde \balpha)}^\infty \E \big( f_j^{(1)} \big)\to 0$ as
$n\to\infty$  by Corollary \ref{cor:neg.tau}. We have now obtained
\eqref{e:Euler.as.conv} and, hence, also \eqref{e:Euler.as.conv0}.

Finally, we can use \eqref{e:EC.faces} to write 
$$
\frac{\E(\chi_n)}{\E(f_k)} = (-1)^k + \frac{\sum_{j=0, \, j \neq
    k}^{n-1}(-1)^j \E(f_j)}{\E(f_k)}. 
$$
With $M(\balpha)$ defined by \eqref{e:M.alpha}, 
$$
\bigg| \frac{\sum_{j=0, \, j \neq k}^{n-1}(-1)^j \E(f_j)}{\E(f_k)}
\bigg| \le \sum_{j=0,\,  j \neq k}^{M(\balpha)-1}\frac{\E(f_j)}{\E(f_k)} + C
\sum_{j=M(\balpha)}^\infty \E(f_j) \to 0, \ \ \ n\to\infty 
$$
by Proposition \ref{p:moment.face.count}   and Corollary
\ref{cor:neg.tau}.  Hence 
$\E(\chi_n)/\E(f_k) \to (-1)^k$,  and 
\eqref{e:SLLN.Euler.char} follows. 
\end{proof}

We now prove the functional central limit theorem for the Euler
characteristic.

\begin{proof}[Proof of \eqref{e:Euler.characteristic.func.conv} in
  Theorem \ref{t:clt.topological.invariants}]
Note, first of all, that for every $M\ge k+1$ the truncated Euler
characteristic
$$
\chi_n^{(M)}(t) = \sum_{j=0}^{M-1} (-1)^j f_j(t)
$$
satisfies, in terms of convergence of the finite-dimensional
distributions, 
$$
\left( \frac{\chi_n^{(M)}(t) - \E(\chi_n^{(M)})}{\sqrt{\Var (f_{k,n})}}, \, t\geq
  0\right)  \Rightarrow \bigl( Z_k(t), \, t\geq 0\bigr). 
$$
This follows from finite-dimensional convergence in Proposition \ref{p:clt.face.counts} and the fact
that by \eqref{e:critical.dom} and Chebyshev's inequality, 
$$
\frac{f_j(t)-\E(f_j)}{\sqrt{\Var(f_k)}} \stackrel{p}{\to} 0, \ \ \
n\to\infty, 
$$
for each $j\not= k$. 

Choosing now $M=M(\balpha)$ defined by \eqref{e:M.alpha}, we have by
Corollary \ref{cor:neg.tau} that 
\begin{equation}  \label{e:Markov.chi}
\P\bigg( \bigg|\frac{\chi_n(t) - \E(\chi_n)}{\sqrt{\Var (f_k)}}
-\frac{\chi_n^{(M(\balpha))} (t)- \E(\chi_n^{(M(\balpha))})}{\sqrt{\Var (f_k)}}   \bigg|
>\epsilon \bigg) \le \frac{2}{\epsilon
  \sqrt{\Var(f_k)}}\sum_{j=M(\balpha)}^\infty \E(f_j) \to 0
\end{equation}
as $n\to\infty$ for any $\epsilon>0$. Therefore, 
$$
\frac{\chi_n(t) - \E(\chi_n)}{\sqrt{\Var (f_k)}} -\frac{\chi_n^{(M(\balpha))}(t) - \E(\chi_n^{(M(\balpha))})}{\sqrt{\Var (f_k)}} \stackrel{p}{\to} 0,
$$
so we have established \eqref{e:Euler.characteristic.func.conv} in
terms of convergence of the finite-dimensional distributions.

\remove{Assuming  \eqref{e:cond.regularity}, we now
establish tightness in the Skorohod $J_1$-topology. We may and will
assume that $T\leq 1/m$, where $m$ is given in
\eqref{e:restriction.epsilon} with $\epsilon=1$ and $T=1$.  Denote 
\begin{equation} \label{e:M1.alpha}
M_1(\balpha)= \min \big\{ i>k: \tau_i(\balpha)< \tau_q(\balpha)\big\},
\end{equation}
and note that $M_1(\balpha)\leq M(\balpha)$ defined in
\eqref{e:M.alpha}. Write
\begin{align} \label{e:split.tightness}
\chi_n(t) =  \sum_{j=0}^{M_1(\balpha)-1} (-1)^j f_j(t)+
  \sum_{j=M_1(\balpha)}^{n-1} (-1)^j f_j(t) := \chi^{(1)}_n(t)
  +\chi^{(2)}_n(t), \ 0\leq t\leq 1.
\end{align}
We start with proving that
\begin{equation} \label{e:chi.2}
\sup_{0\leq t\leq 1/m} \left| \frac{\chi^{(2)}_n(t) - \E(\chi^{(2)}_n)}{\sqrt{\Var
      (f_k)}}\right| \stackrel{p}{\to} 0. 
\end{equation}  
Indeed, we have
\begin{align*}
\E \sup_{0\leq t\leq 1/m} \left| \frac{\chi^{(2)}_n(t) - \E(\chi^{(2)}_n)}{\sqrt{\Var
      (f_k)}}\right| 
\leq & C \sum_{j=M_1(\balpha)}^{n-1} \frac{\E \sup_{0\leq t\leq 
   1/m}f_j(t)}{n^{2\tau_k(\al)-\tau_q(\al)}} \\
 \leq & 2C \sum_{j=M_1(\balpha)}^{n-1} \frac{\E
        f_j}{n^{2\tau_k(\al)-\tau_q(\al)}} \to 0
\end{align*}
as $n\to\infty$ by Corollary \ref{cor:neg.tau} since
$$
\E f_j = Cn^{\tau_j(\balpha)} = o\bigl( n^{\tau_q(\balpha)}\bigr) 
= o\bigl( n^{2\tau_k(\al)-\tau_q(\balpha)}\bigr) 
$$
for $j\geq M_1(\balpha)$. Therefore, \eqref{e:chi.2} follows, and so it
remains to prove tightness of the process $\bigl( \chi^{(1)}_n(t),\, 
0\leq t\leq T\bigr)$. 
}

Assuming \eqref{e:cond.regularity} and \eqref{e:sharp.drop.k+1}, we now
establish tightness in the Skorohod $J_1$-topology. 
Denote 
$$
M_1(\balpha)= \min \big\{ i>k: \tau_i(\balpha)< \tau_q(\balpha)\big\}. 
$$
Fix $T>0$ and choose $m$ so that $T/m \le a /4,$ where $a$ is the constant from \eqref{e:assumption.Gi}. Recall once again the notation $p_i^{(1)}=\P\big( \sup_{0\le t \le T/m} \Delta_i(t)=1 \big)$, $i\ge 1$, so that $p_i^{(1)} \le p_i (2-p_i) \le n^{-\tilde \al_i},$ where $\tilde \balpha = (\tilde \al_1, \tilde \al_2, \dots)$ is as defined below \eqref{e:alpha.tilde.rate}. Note that $M_1(\balpha) \le  M(\balpha) \leq  M(\tilde \balpha) <\infty$, where $M(\balpha)$ and $M(\tilde \balpha)$ are as defined in \eqref{e:M.alpha}.
Recall also that for $j \ge q$, $f_j^{(1)}$ is the $j$-face counts in $X([n], \bp^{(1)})$, such that $f_j^{(1)} \stackrel{sd}{\ge} \sup_{0 \le t \le T/m} f_j(t)$.  Write
\begin{align*}
\chi_n(t) &=  \sum_{j=0}^{M_1(\balpha)-1} (-1)^j f_j(t)+
  \sum_{j=M_1(\balpha)}^{n-1} (-1)^j f_j(t) =: \chi^{(1)}_n(t)
  +\chi^{(2)}_n(t), \ 0\leq t\leq T.
\end{align*}
We start with proving that, as $n\to\infty$, 
$$
\frac{\sup_{0\leq t\leq T} \big|  \chi^{(2)}_n(t) - \E(\chi^{(2)}_n) \big|}{\sqrt{\Var
      (f_k)}} \stackrel{p}{\to} 0. 
$$
By stationarity, it suffices to show that 
\begin{equation}  \label{e:chi2and3}
\frac{\sup_{0\leq t\leq T/m} \big|  \chi^{(2)}_n(t) - \E(\chi^{(2)}_n) \big|}{\sqrt{\Var
      (f_k)}} \stackrel{p}{\to} 0. 
\end{equation}
Let $\epsilon > 0$ be arbitrary. Then, by Markov's inequality, for all sufficiently large $n,$
\begin{align}
\P \big( \sup_{0\le t \le T/m} \big| \chi_n^{(2)}-\E(\chi_n^{(2)}) \big| > \epsilon \sqrt{\Var(f_k)} \big) &\le \frac{2}{\epsilon \sqrt{\Var(f_k)}}\, \sum_{j=M_1(\balpha)}^{n-1} \E \big[ \sup_{0\le t \le T/m} f_j(t) \big] \label{e:Markov.chi2}\\
&\le \frac{2}{\epsilon \sqrt{\Var(f_k)}}\, \sum_{j=M_1(\balpha)}^{n-1} \E(f_j^{(1)}) \notag \\
&\le \frac{2}{\epsilon} \sum_{j=M_1(\balpha)}^{M(\tilde \balpha)-1} \frac{\prod_{i=q}^j 2^{\binom{j+1}{i+1}}\E(f_j)}{\sqrt{\Var(f_k)}} +\frac{2}{\epsilon} \sum_{j=M(\tilde \balpha)}^\infty \E(f_j^{(1)}), \notag
\end{align}
where the last inequality is due to Proposition 3.1, together with the fact that $p_i^{(1)} \le 2p_i$. The second term vanishes because $\sum_{j=M(\tilde \balpha)}^\infty \E(f_j^{(1)}) \to 0$, as $n\to\infty,$ by Corollary 3.4. On the other hand, the first vanishes since, by \eqref{e:sharp.drop.k+1}, 
$$
\E(f_j) \le n^{\tau_j(\balpha)} \le n^{\tau_{k+1}(\balpha)} = o \big( n^{\tau_k(\balpha) -\tau_q(\balpha)/2} \big) = o \big( \sqrt{\Var(f_k)} \big), \ \ n\to\infty. 
$$
Now \eqref{e:chi2and3} follows as desired, and so it remains to prove tightness of the process $\big( \chi_n^{(1)}(t), \, 0 \le t \le T \big)$.
To this aim, 
it is enough to show the existence of $B\in (0,\infty)$ such that 
\begin{equation}  \label{e:tightness.Thm13.5.Euler}
\frac{\E \Big[ \big( \chi^{(1)}_n(t)-\chi^{(1)}_n(s) \big)^2 \big(
  \chi^{(1)}_n(s)-\chi^{(1)}_n(r) \big)^2  \Big]}{n^{4\tau_k(\al)-2\tau_q(\al)}}
\le B (t-r)^{1+\gamma}  
\end{equation}
for all $0\leq r\leq s \leq t\leq T$ and $n\ge 1$. In the course of the proof, we will also establish
\eqref{e:tightness.Thm13.5.face.counts} needed for the tightness in
Proposition \ref{p:clt.face.counts}. 

We begin by setting up the notation. 
For $q+1 \le j_1, j_2 <M_1(\balpha)$ and $0 \le r \le s \le t \le T,$ denote 
$$
F_{j_1,j_2}(t,s,r) := \E \Big[ \big( f_{j_1}(t)-f_{j_1}(s) \big)^2 \big( f_{j_2}(s)-f_{j_2}(r) \big)^2  \Big]. 
$$
Consider a potential subcomplex $\bar \sigma$ in $[n]$ consisting of the 4 simplices $\sigma_1,
\sigma_2, \sigma_3, \sigma_4$ and their faces, with
$|\sigma_1|=|\sigma_2|=j_1+1$, $|\sigma_3|=|\sigma_4|=j_2+1$,  and let 
$$
a_{ij} = |\sigma_i \cap \sigma_j|, \ \ 1 \le i < j \le 4,  \ \ \ a_{ijk} = |\sigma_i \cap \sigma_j \cap \sigma_k|, \ \ 1 \le i < j < k \le 4, 
$$
$$
a_{1234} = |\sigma_1 \cap \sigma_2 \cap \sigma_3 \cap \sigma_4|. 
$$
The number of $i$-faces in $\bar \sigma$ is 
\begin{align*}
\text{comb}_i(\bar \sigma) &:= 2\binom{j_1+1}{i+1} + 2 \binom{j_2+1}{i+1} - \binom{a_{12}}{i+1} - \binom{a_{13}}{i+1} - \binom{a_{14}}{i+1} - \binom{a_{23}}{i+1} \\
&\quad - \binom{a_{24}}{i+1} - \binom{a_{34}}{i+1} +
                                                                                                                                                                       \binom{a_{123}}{i+1} + \binom{a_{124}}{i+1} + \binom{a_{134}}{i+1} + \binom{a_{234}}{i+1} - \binom{a_{1234}}{i+1};  
\end{align*}
it depends only on $j_1, j_2$, and $\ba =(a_{12}, \dots, a_{1234})$. 
We let 
\begin{align}
\Psi(\ba, \balpha)&:=\tau_{a_{12}-1}(\balpha) + \tau_{a_{13}-1}(\balpha) + \tau_{a_{14}-1}(\balpha) + \tau_{a_{23}-1}(\balpha) + \tau_{a_{24}-1}(\balpha) + \tau_{a_{34}-1}(\balpha) \label{e:psi.ba.al} \\
&\quad  - \tau_{a_{123}-1}(\balpha)- \tau_{a_{124}-1}(\balpha)- \tau_{a_{134}-1}(\balpha)- \tau_{a_{234}-1}(\balpha) +  \tau_{a_{1234}-1}(\balpha) \notag
\end{align}
(with $\tau_{-1}(\balpha)\equiv0$). 
By independence, 
\begin{align}
F_{j_1, j_2}(t,s,r) &= \sum_{\bar \sigma \subset \Xi(j_1, j_2)} \E \Big[ \big( \xi_{\sigma_1}(t) - \xi_{\sigma_1}(s) \big) \big( \xi_{\sigma_2}(t) - \xi_{\sigma_2}(s) \big)  \big( \xi_{\sigma_3}(s) - \xi_{\sigma_3}(r) \big)  \big( \xi_{\sigma_4}(s) - \xi_{\sigma_4}(r) \big)\Big] \notag\\
&=: \sum_{\bar \sigma \subset \Xi(j_1, j_2)} \E  \big[ g(t,s,r; \bar \sigma)  \big], \label{e:expectation.g}
\end{align}
with the summation restricted to the set 
\begin{align*}
\Xi(j_1,j_2) &= \big\{  \bar \sigma = (\sigma_1, \dots, \sigma_4): \, |\sigma_1|=|\sigma_2|=j_1+1, \,  |\sigma_3|=|\sigma_4|=j_2+1, \\ 
& \qquad \quad \text{and } (\sigma_1,\dots,\sigma_4) \text{ satisfies at least one of the conditions in \eqref{e:def.as} below} \big\}:  
\end{align*}
%
\begin{align}
&\text{(i)} \, a_{12}\ge q+1, a_{34}\ge q+1, \ \  \text{(ii)}\, a_{13}\ge q+1, a_{24}\ge q+1,  \ \ \text{(iii)} \, a_{14}\ge q+1, a_{23} \ge q+1,  \notag \\
&\text{(iv)}\, a_{12} \ge q+1, a_{13}\ge q+1, a_{14} \ge q+1, \ \ \text{(v)}\, a_{12} \ge q+1, a_{23}\ge q+1, a_{24} \ge q+1, \label{e:def.as} \\
&\text{(vi)}\, a_{13} \ge q+1, a_{23}\ge q+1, a_{34} \ge q+1, \ \ \text{(vii)}\, a_{14} \ge q+1, a_{24}\ge q+1, a_{34} \ge q+1.  \notag
\end{align}
%
%
Indeed, if none of the conditions in \eqref{e:def.as} holds, then the corresponding term
in \eqref{e:expectation.g} vanishes by independence and stationarity.

Our goal is to bound the expectation  $\E\big[ g(t,s,r; \bar \sigma)
\big]$ in \eqref{e:expectation.g}. Note that $g(t, s, r; \bar \sigma) \in \{-1, 0, +1\}.$ Hence, for $g(t, s, r; \bar \sigma)$ not to vanish, 
every $i$-face of the simplex $\sigma_1$ must exist
either at time $s$ or at time $t$, $i=q,\ldots, j_1$, and the same is true for the
simplex $\sigma_2$. Similarly, every $i$-face of the simplex $\sigma_3$ must exist
either at time $r$ or at time $s$, $i=q,\ldots, j_2$, and the same is true for the
simplex $\sigma_4$. The probability that this happens is bounded from
above by 
\begin{equation} \label{onetime.prod}
16 \prod_{i=q}^{j_1\vee j_2} p_i^{\text{comb}_i(\bar \sigma)},
\end{equation} 
where we take into account only the first (smallest) time a face
exists if it is required to exist multiple times. Additionally, at
least one face of the complex spanned by the simplices
$\sigma_1,\sigma_2$ must switch from existence to non-existence, or vice versa,
between times $s$ and $t$, and at least one face of the complex spanned
by the simplices 
$\sigma_3,\sigma_4$ must switch from existence to non-existence, or vice versa,
between times $r$ and $s$. This may be the same face or two different
faces. Let us denote the corresponding (non-disjoint) events by $A_1$
and $A_2$. Consider the event $A_1$ first. The number of possible
faces that can change their status does not exceed the total number of
faces in $\bar\sigma$, which is, in turn, bounded by
$2^{2(j_1+j_2)}$. For such an $i$-face the probability $p_i$ in
\eqref{onetime.prod}  will be replaced by one by of following two
probabilities: 
$$
\P\bigl( \Delta_i(r)=0, \Delta_i(s)=1, \Delta_i(t)=0\bigr)
$$
and
$$
\P\bigl( \Delta_i(r)=1, \Delta_i(s)=0,\Delta_i(t)=1\bigr),
$$
both of which are bounded by $(2c/a)p_i(t-r)^{1+\gamma}$ by Lemma
\ref{l:triples}. Therefore,
$$
\P(A_1)\leq C 2^{2(j_1+j_2)}(t-r)^{1+\gamma} \prod_{i=q}^{j_1\vee j_2}
p_i^{\text{comb}_i(\bar \sigma)}.
$$
Considering the event $A_2$ now, we see that the number of possible
pairs of faces that can change their status does not
exceed $2^{4(j_1+j_2)}$. For each such a pair of an $i_1$-face and an
$i_2$-face, the product $p_{i_1}p_{i_2}$  in
\eqref{onetime.prod}  will be, up to renaming, replaced by
$$
\P\bigl( \Delta_{i_1}(r)=1,\Delta_{i_1}(s)=0\bigr)
\P\bigl( \Delta_{i_2}(s)=1,\Delta_{i_2}(t)=0\bigr),
$$
or similar expressions obtained by flipping $1$s and $0$s. By Lemma \ref{l:cond.prob.Delta}, any such expression is bounded by
$$
p_{i_1}p_{i_2} \left(\frac{2}{a}\right)^2 (t-r)^2.
$$
Since $\gamma\leq 1$, we conclude that
$$
\P(A_2)\leq C 2^{4(j_1+j_2)} (t-r)^{1+\gamma}\prod_{i=q}^{j_1\vee j_2}
p_i^{\text{comb}_i(\bar \sigma)}, 
$$
and so
\begin{equation}  \label{e:comb.prod}
\E  \big[ |g(t,s,r; \bar \sigma) | \big]  \leq C 2^{4(j_1+j_2)}
(t-r)^{1+\gamma}\prod_{i=q}^{j_1\vee j_2} 
p_i^{\text{comb}_i(\bar \sigma)}.  
\end{equation} 

Substituting this back into \eqref{e:expectation.g},  we obtain 
\begin{align*}
F_{j_1, j_2}(t,s,r) &\le C 2^{4(j_1+j_2)}   (t-r)^{1+\gamma} \sum_{\bar \sigma \in \Xi(j_1, j_2)} \prod_{i=q}^{j_1 \vee j_2} p_i^{\text{comb}_i(\bar \sigma)} \\
&=  C 2^{4(j_1+j_2)}   (t-r)^{1+\gamma} \sum_{\ba \in \A} \sum_{\bar \sigma \in \Xi(j_1, j_2)}
                                                                                                                                                                   \hspace{-10pt} \one \big\{
                 |\sigma_1 \cap \sigma_2|= a_{12}, \,   |\sigma_1 \cap
                 \sigma_3|=a_{13}, \dots, \\
                 &\qquad \qquad \qquad \qquad \qquad \qquad \qquad \qquad \qquad |\sigma_1 \cap \sigma_2 \cap \sigma_3 \cap \sigma_4| =a_{1234} \big\} \prod_{i=q}^{j_1 \vee j_2}
         p_i^{\text{comb}_i(\bar \sigma)},  
\end{align*}%
where $\A$ is the collection of $\ba = (a_{12}, \dots, a_{1234})$
satisfying   at least one of the conditions in \eqref{e:def.as}. 
Note that $\text{comb}_i(\bar \sigma)$ depends only on $\ba$, and for
any $\ba$, 
\begin{align*}
&\sum_{\bar \sigma \in \Xi(j_1, j_2)} \hspace{-10pt} \one \big\{
                 |\sigma_1 \cap \sigma_2|= a_{12}, \,   |\sigma_1 \cap
                 \sigma_3|=a_{13}, \dots, |\sigma_1 \cap \sigma_2 \cap \sigma_3 \cap \sigma_4| =a_{1234} \big\} \le n^{\text{comb}_0(\bar \sigma)}. 
\end{align*}
Since 
$$
n^{\text{comb}_0(\bar \sigma)} \prod_{i=q}^{j_1\vee j_2} p_i^{\text{comb}_i(\bar \sigma)} = n^{2(\tau_{j_1}(\balpha)+\tau_{j_2}(\balpha))-\Psi(\ba, \balpha)} 
$$
with $\Psi(\ba, \balpha)$  given in \eqref{e:psi.ba.al}, we obtain 
\begin{align}
F_{j_1, j_2}(t,s,r) &\le  C 2^{4(j_1+j_2)}   (t-r)^{1+\gamma}
\sum_{\ba \in \A} n^{2(\tau_{j_1}(\balpha)+\tau_{j_2}(\balpha))-\Psi(\ba, \balpha)}. \label{e:final.Fj1j2} 
\end{align}

We proceed with the following lemma. 
\begin{lemma}  \label{l:ratio.4thmoment}
For $q+1 \le j_1, j_2 < M_1(\alpha)$ and $\ba = (a_{12}, \dots, a_{1234})\in \A$, we have 
\begin{equation}  \label{e:ratio.4thmoment.1}
\frac{n^{2(\tau_{j_1}(\balpha)+\tau_{j_2}(\balpha))-\Psi(\ba, \balpha)}}{n^{4\tau_k(\balpha)-2\tau_q(\balpha)}} \le 1. 
\end{equation}
\end{lemma}
\begin{proof}
Notice that 
\begin{align*}
D &:=  2(\tau_{j_1}(\balpha)+\tau_{j_2}(\balpha))-\Psi(\ba, \balpha) \\
&\leq
  2(\tau_{j_1}(\balpha)+\tau_{j_2}(\balpha))   
  - \tau_{a_{12}-1}(\balpha)-\tau_{a_{13}-1}(\balpha)- \tau_{a_{14}-1}(\balpha) -
  \tau_{a_{23}-1}(\balpha) - \tau_{a_{24}-1}(\balpha) -
      \tau_{a_{34}-1}(\balpha)  \\
  &\quad +\tau_{a_{123}-1}(\balpha)+ \tau_{a_{124}-1}(\balpha)+
     \tau_{a_{134}-1}(\balpha)+ \tau_{a_{234}-1}(\balpha),
\end{align*}
and, by the choice of $j_1,j_2,$ all the terms $\tau_\cdot(\balpha)$ in
the right hand side are non-negative. 
Since the sequence $(\tau_i(\balpha), \, i\ge -1)$ is unimodal -- it
increases until $i=k$ and then decreases - we have 
\begin{align} \label{e:unim.bounds}
&\tau_{a_{12}-1}(\balpha)\geq \min\bigl( \tau_{j_1}(\balpha), \tau_{a_{123}-1}(\balpha)\vee
  \tau_{a_{124}-1}(\balpha)\bigr), \\
&\tau_{a_{13}-1}(\balpha)\geq \min\bigl( \tau_{j_1}(\balpha) \vee
  \tau_{j_2}(\balpha), \tau_{a_{123}-1}(\balpha)\vee
                                        \tau_{a_{134}-1}(\balpha)\bigr), \notag\\
&\tau_{a_{14}-1}(\balpha)\geq \min\bigl( \tau_{j_1}(\balpha) \vee
  \tau_{j_2}(\balpha), \tau_{a_{124}-1}(\balpha)\vee
                                        \tau_{a_{134}-1}(\balpha)\bigr), \notag \\
&\tau_{a_{23}-1}(\balpha)\geq \min\bigl( \tau_{j_1}(\balpha) \vee
  \tau_{j_2}(\balpha), \tau_{a_{123}-1}(\balpha)\vee
           \tau_{a_{234}-1}(\balpha)\bigr),  \notag \\
&\tau_{a_{24}-1}(\balpha)\geq \min\bigl( \tau_{j_1}(\balpha) \vee
  \tau_{j_2}(\balpha), \tau_{a_{124}-1}(\balpha)\vee
           \tau_{a_{234}-1}(\balpha)\bigr),  \notag \\                                                                                                                                                    
&\tau_{a_{34}-1}(\balpha)\geq \min\bigl(  
  \tau_{j_2}(\balpha), \tau_{a_{134}-1}(\balpha)\vee
           \tau_{a_{234}-1}(\balpha)\bigr).  \notag 
\end{align}  
Since $\ba \in \A$, at least one of the 6 conditions in
\eqref{e:def.as} holds. We will consider in detail what happens under
condition $(i)$; the situation under the other conditions is similar. 

Under condition $(i)$ in \eqref{e:def.as} the first and the last
bounds in \eqref{e:unim.bounds} are supplemented by the bounds
$ \tau_{a_{12}-1}(\balpha)\geq    \tau_q(\balpha)$, 
 $ \tau_{a_{34}-1}(\balpha)\geq    \tau_q(\balpha)$. 
We now use the remaining 4 inequalities in
\eqref{e:unim.bounds}. Note that $\tau_{a_{13}-1}(\balpha)$ ``kills"
(i.e., is at least as large as) $\tau_{j_1}(\balpha)$, 
  $\tau_{j_2}(\balpha)$ or $\tau_{a_{123}-1}(\balpha)$. Similarly,  $\tau_{a_{14}-1}(\balpha)$ ``kills"
  $\tau_{j_1}(\balpha)$,   $\tau_{j_2}(\balpha)$ or
  $\tau_{a_{134}-1}(\balpha)$. Further, $\tau_{a_{23}-1}(\balpha)$ ``kills"
  $\tau_{j_1}(\balpha)$,   $\tau_{j_2}(\balpha)$ or
  $\tau_{a_{234}-1}(\balpha)$. Finally, $\tau_{a_{24}-1}(\balpha)$ ``kills"
  $\tau_{j_1}(\balpha)$,   $\tau_{j_2}(\balpha)$ or
  $\tau_{a_{124}-1}(\balpha)$. This leaves 4 non-negative terms in the
  upper bound for $D$, neither of which exceeds $\tau_k(\balpha)$, so 
  $D\leq 4\tau_k(\balpha) -2\tau_q(\balpha)$, as required. 
\end{proof}

Since $\A$ is parameterized by the $11$ variables $a_{12}, \ldots, a_{1234},$ its cardinality does not exceed $(j_1+j_2+1)^{11}.$ Hence, by  Lemma \ref{l:ratio.4thmoment} and  \eqref{e:final.Fj1j2}  
\begin{align*}
&  \frac{\E \Big[ \big( \chi^{(1)}_n(t)-\chi^{(1)}_n(s) \big)^2 \big(
                 \chi^{(1)}_n(s)-\chi^{(1)}_n(r) \big)^2
                 \Big]}{n^{4\tau_k(\balpha)-2\tau_q(\balpha)}}    \le
                 \sum_{j_1=q+1}^{M_1(\balpha)-1}\sum_{j_2=q+1}^{M_1(\balpha)-1} 
                 \frac{F_{j_1,
                 j_2}(t,s,r)}{n^{4\tau_k(\balpha)-2\tau_q(\balpha)}}
   \leq B (t-r)^{(1+\gamma)}
\end{align*}
for some $0 < B<\infty$, as required for
\eqref{e:tightness.Thm13.5.Euler}. 
\end{proof}

\section{Proofs of the limit theorems  for the Betti numbers in the critical
  dimension }  \label{s:proofs.betti} 

Once again, we start with the strong law of large numbers. 

\begin{proof}[Proof of \eqref{e:SLLN.betti} in Theorem \ref{t:SLLN}]
For $0<T<\infty$, we have to demonstrate that 
$$
\sup_{0 \le t \le T}\frac{\big| \beta_k(t)-\E(f_k) \big|}{\E(f_k)} \to 0 \ \ \text{a.s.}
$$
By the Morse inequalities 
$$
f_k(t)-f_{k+1}(t)-f_{k-1}(t) \le \beta_k(t) \le f_k(t), 
$$
we have 
$$
\big| \beta_k(t)-\E(f_k)  \big| \le \big| f_k(t)-\E(f_k) \big| +
f_{k+1}(t) + f_{k-1}(t). 
$$
By \eqref{e:first.claim.SLLN} with $j=k$, it is enough to prove that as $n\to\infty$, 
$$
\sup_{0\le t \le T}\frac{f_{k+1}(t)}{\E(f_k)} \to 0 \ \ \text{a.s.} \ \ \text{and } \sup_{0\le t \le T}\frac{f_{k-1}(t)}{\E(f_k)} \to 0 \ \ \text{a.s.} 
$$
This is, however, an immediate conclusion of
\eqref{e:first.claim.SLLN} with $j=k\pm 1$, since by
Proposition \ref{p:moment.face.count}, 
$$
\lim_{n\to\infty} \frac{\E( f_{k+1})}{\E( f_k)} = \lim_{n\to\infty} \frac{\E( f_{k-1})}{\E( f_k)} = 0.
$$
\end{proof}

We continue with the functional central limit theorem for Betti
numbers. 

\begin{proof}[Proof of \eqref{e:betti.func.conv} in
Theorem \ref{t:clt.topological.invariants}] For convenience, we drop
the subscript $n$ in expressions such as $\beta_{j,n}$ for the
duration of the proof. We start with
introducing some terminology related to the connectivity of a
simplicial complex. It is analogous to the terminology used in
\cite{kahle:2009} and \cite{fowler:2019}. An $\ell$-dimensional simplicial complex
$X$,   is called \textit{pure} if every face of $X$ is contained
in an $\ell$-face.  
A   simplicial complex $K$ is said to be
\textit{strongly connected} of order $\ell$ if the following two
conditions hold:
\begin{itemize}
\item The $\ell$-skeleton of $K$ is pure. 
\item Every pair of $\ell$-faces $\sigma, \tau \in K$, can be connected by a sequence of $\ell$-faces, 
$$
\sigma=\sigma_0, \sigma_1, \dots, \sigma_{j-1}, \sigma_j =  \tau
$$
for some $j\geq 1$, such that $\text{dim}(\sigma_i \cap \sigma_{i+1}) = \ell-1$, $0 \le i \le j-1$. 
\end{itemize}
In this case, we will simply say that $K$ is an $\ell$-strongly
connected simplicial complex. 
Note that the dimension of $K$ itself may be greater than $\ell$. 
We call an $\ell$-strongly connected subcomplex $K$ of $X$
\textit{maximal} if there is no other $\ell$-strongly connected
subcomplex $K' \supset K$.  We start with a useful estimate similar to
the computation in \cite{fowler:2019}, p.117. 

\begin{lemma}  \label{l:formula.fowler} 
Let $K$ be a $(k+1)$-strongly connected simplicial complex 
on $j\geq k+3$ vertices with a non-zero $(k+1)$-st Betti number. Then,
for $\sigma \subset [n]$ with $|\sigma|=j$,  
\begin{equation*} 
\P(\text{the restriction of $X([n],\bp)$ to $\sigma$  is isomorphic to } K) \le  j!\prod_{i=q}^{k+1}  p_i^{\binom{k+3}{i+1}} \Big( \prod_{i=q}^{k+1} p_i^{\binom{k+1}{i}} \Big)^{j-k-3}.
\end{equation*}
\end{lemma}
\begin{proof}
The argument consists of estimating the number of faces of different
dimensions $K$ has to contain. We start by denoting by $m$ the number
of the   $(k+1)$-faces in $K$. We order these faces as
follows. Fix an arbitrary $(k+1)$-cycle in $K$ and choose any
$(k+1)$-face from this cycle to be $f_1$. Since $K$ is
$(k+1)$-strongly connected, we can order the rest of the
$(k+1)$-faces in the order $f_1, \ldots, f_m$ such that each $f_p$, $p>1,$ has a $k$-dimensional intersection with at least one $f_q$ with $q<p$. This
ordering of the $(k+1)$-faces  induces an ordering on the vertices in
$K$, as follows.  First, let $v_1,\dots,v_{k+2}$ be the vertices, chosen in an arbitrary order, in the
support of $f_1.$
Each vertex after $v_{k+2}$ corresponds to the addition of a $(k+1)$-face $f_\ell;$ in that, it lies in the support of $f_{\ell}$ but is not contained in $f_1\cup \dots \cup f_{\ell-1}$. Since each vertex of $K$
belongs to some $(k+1)$-face, we obtain, in this way, an ordering
$v_{k+3}, \dots, v_j$ of all remaining vertices in $K$. Note at this
point that each vertex after $v_{k+2}$, for each $1 \le i \le k+1$, 
 is a vertex of $\binom{k+1}{i}$ of $i$-faces of some new  
  $(k+1)$-face $f_\ell$ being considered at that point. We  let 
$$
c = \max\{ k+3 \le m \le j: v_m \text{ is a vertex of the initially fixed } (k+1)\text{-cycle} \}
$$
and note that $c$ is well defined since the cycle must contain at
least $k+3$ vertices. The corresponding vertex $v_c$ is, actually,
contained in at least $k+2$ faces of dimension $k+1$, just as other vertices in the initially fixed $(k+1)$-cycle.
Furthermore, $v_c$ is contained in the fewest number of $i$-faces if it is a part of exactly $k + 2$ faces of dimension $k + 1.$ The latter occurs when, excluding $v_c,$ there are precisely $k + 2$ other vertices in this cycle and they together form a $(k +1)$-face. Therefore, when $v_c$ entered our
enumeration of the vertices, for each $1\le i \le k+1$, it was 
a vertex of at least $\binom{k+2}{i}$ new $i$-faces in $K$. We  now
see that for each $1\le i \le k+1$, 
\begin{itemize}
\item   $f_1$ contains $\binom{k+2}{i+1}$ distinct $i$-faces in $K$; 
\item each vertex in $\{ v_{k+3}, \dots, v_j \}\setminus \{ v_c \}$
  corresponds to   $\binom{k+1}{i}$ new distinct $i$-faces in $K$; 
\item $v_c$ corresponds to   at least $\binom{k+2}{i}$ new distinct $i$-faces
  in $K$. 
\end{itemize}
Therefore, for each $1\le i \le k+1$,  $K$ contains  at least 
$$
\binom{k+2}{i+1} + (j-k-3)\binom{k+1}{i} + \binom{k+2}{i} = \binom{k+3}{i+1} +  (j-k-3)\binom{k+1}{i} 
$$
$i$-faces. Finally, since there are $j!$ ways of ordering vertices in $\sigma$, we get the assertion of the lemma.
\end{proof}

By \eqref{e:Euler.characteristic}, the already established convergence in
\eqref{e:Euler.characteristic.func.conv} tells us that  
\begin{equation}  \label{e:rephrase.to.Betti}
\left( \frac{\sum_{j=0}^{n-1} (-1)^j \beta_j (t)- \E \big(
    \sum_{j=0}^{n-1} (-1)^j \beta_j \big)}{\sqrt{\Var
      (f_{k})}}, \, t\geq 0\right)   \Rightarrow \bigl(Z_k(t), \,
t\geq 0\bigr)
\end{equation}
in finite dimensional distributions. In order to prove convergence in
finite dimensional distributions in \eqref{e:betti.func.conv},  we
need to show that all (normalized) Betti numbers except that of
critical dimension are asymptotically negligible in
\eqref{e:rephrase.to.Betti}. Proposition
\ref{p:vanishing.lower.order.betti} in the Appendix 
shows negligibility of the Betti numbers in dimension smaller than the
critical dimension. Together with \eqref{e:rephrase.to.Betti}, this
gives that
$$
\left( \frac{\sum_{j=k}^{n-1} (-1)^j \beta_j(t) - \E \big(
    \sum_{j=k}^{n-1} (-1)^j \beta_j \big)}{\sqrt{\Var (f_k)}}, \,
  t\geq 0\right)  \Rightarrow  \bigl( Z_k(t), \,
t\geq 0\bigr). 
$$
Furthermore, by repeating the same argument as in \eqref{e:Markov.chi}, along with an obvious bound $\beta_j \le f_j$, we obtain that 
$$
\left( \frac{\sum_{j=M(\balpha)}^{n-1} (-1)^j \beta_j(t) - \E \big( \sum_{j=M(\balpha)}^{n-1} (-1)^j \beta_j \big)}{\sqrt{\Var (f_k)}}, \,
  t\geq 0\right)  \to  {\bf 0}, 
$$
in finite-dimensional distributions, where $M(\balpha)$ is defined in \eqref{e:M.alpha} and $\bf 0$ is the constant zero process. Hence, we can conclude that 
\begin{equation}  \label{e:k.and.higher}
\left( \frac{\sum_{j=k}^{M(\balpha)-1} (-1)^j \beta_j(t) - \E \big( \sum_{j=k}^{M(\balpha)-1} (-1)^j \beta_j \big)}{\sqrt{\Var (f_k)}}, \,
  t\geq 0\right)   \Rightarrow \bigl(  Z_k(t),  \,
t\geq 0\bigr),\ \ \ n\to\infty, 
\end{equation}
in finite-dimensional distributions.

Note that if $M(\balpha)=k+1$, then \eqref{e:betti.func.conv} is
automatic, so only the case $M(\balpha)>k+1$ needs to be
considered. 
It is, of course, sufficient to show that 
for any $j=k+1, \ldots, M(\balpha)-1$, 
$\Var (\beta_j)$ is negligible relative to $\Var (f_k)$ as
$n\to\infty$.  We will consider in detail the case $M(\balpha)=k+2$, and  prove negligibility of the variance  of
$\beta_{k+1}$. If $M(\balpha)>k+2$, the  
higher-order Betti numbers can be treated in a similar way.

Our argument relies on an explicit representation of $\beta_{k+1}(t)$ given by
\begin{equation}  \label{e:beta.k+1}
\beta_{k+1}(t) = \beta_{k+1} \big( X([n],\bp; t) \big) = \sum_{j=k+3}^n \sum_{r\ge 1}  \sum_{\sigma \subset [n], \, |\sigma|=j}r \etasigjrko(t), 
\end{equation}
where $\etasigjrko(t)$ is the indicator function of the event that
$\sigma$ forms a maximal $(k+1)$-strongly connected subcomplex
$X(\sigma, \bp;t)$,  such that $\beko \big( X(\sigma, \bp;t)
\big)=r$. See  Proposition \ref{p:betti.representation} for a formal derivation of \eqref{e:beta.k+1}. 
We often omit superscripts from the indicator if the context is clear
enough. Note that the second sum over $r \ge 1$ is a sum of at most
$\binom{j}{k+2}$ terms, because $\beta_{k+1} \big( X(\sigma, \bp;t)
\big)$ is bounded by the number of $(k+1)$-faces of $\sigma$, which
itself is bounded by $\binom{j}{k+2}$.  

As $M(\balpha)=k+2$, it follows that $\tau_{k+1}(\balpha)
> 0$, and we can find a positive integer $D$ such that  
\begin{equation}  \label{e:constraint.D}
D > \frac{k+2+\tau_{k+1}(\balpha)}{\psi_{k+1}(\balpha)-1} >0, 
\end{equation}
and we use it to define a truncated version of the representation of 
the Betti number in \eqref{e:beta.k+1} as 
$$
\tilde{\beta}_{k+1}(t)  = \tilde{\beta}_{k+1} \big( X([n],\bp;t) \big) = \sum_{j=k+3}^{D+k+1} \sum_{r\ge 1} \sum_{\sigma \subset [n], \, |\sigma|=j}r \etasigjrko(t). 
$$ 
As before, we write $\tilde{\beta}_{k+1} := \tilde{\beta}_{k+1}(0)$ and $\etasigjrko := \etasigjrko(0)$. 
We claim that 
\begin{equation}  \label{e:beta.k.and.beta.tilde}
\left( \frac{\beta_k(t)-\E(\beta_k)}{\sqrt{\Var(f_k)}} -
  \frac{\tilde{\beta}_{k+1}(t)-\E(\tilde{\beta}_{k+1})}{\sqrt{\Var(f_k)}}, \,
  t\geq 0\right)   \Rightarrow \bigl( 
 Z_k(t),   \, t\geq 0\bigr), \ \ \ n\to\infty
\end{equation}
in finite-dimensional distributions. Indeed, by \eqref{e:k.and.higher}
with $M(\balpha)=k+2$, it is enough to prove that 
$\E (\beta_{k+1}-\tilde{\beta}_{k+1})
\to 0$, $n\to\infty$. 
Since the sum over $r\ge 1$ in \eqref{e:beta.k+1} contains at most
$\binom{j}{k+2}$ terms,  
\begin{align*}
\E(\beta_{k+1}-\tilde{\beta}_{k+1}) \le&\, \E \bigg[ \sum_{j=D+k+2}^n
                                         \binom{j}{k+2}\, 
                                         \sum_{\sigma \subset [n], \,
                                         |\sigma|=j} \sum_{r\ge 1}
                                         \eta_\sigma^{(j,r,k+1)} \biggr]  \\
\le&\, \E \bigg[ \, n^{k+2} \sum_{j=D+k+2}^n
                                         \sum_{\sigma \subset [n], \,
                                         |\sigma|=j} \sum_{r\ge 1}
                                         \eta_\sigma^{(j,r,k+1)} \biggr]. 
\end{align*}
Whenever a $(k+1)$-strongly
connected subcomplex 
is formed on $j  \ge D+k+2$ vertices, it contains
a further $(k+1)$-strongly connected subcomplex 
on exactly $D+k+2$ vertices.
Furthermore, no two different such
maximal subcomplexes can contain the same $(k+1)$-strongly connected
subcomplex on $D+k+2$ vertices. Therefore,  
\begin{align*}
\E(\beta_{k+1}-\tilde{\beta}_{k+1}) \leq{} & n^{k + 2}  \binom{n}{D + k + 2} \sum_{K: |K|=D+k+2} \P(\sigma_{D+k+2} \text{ is isomorphic to } K) \\
\le {} &    \frac{n^{D+2k+4}}{(D + k + 2)!} \sum_{K: |K|=D+k+2} \P(\sigma_{D+k+2} \text{ is isomorphic to } K), 
\end{align*}
where $\sigma_{D+k+2}$ is the restriction of the complex to fixed
$D+k+2$ vertices, and  the sum above is taken over all isomorphism
classes of 
$(k+1)$-strongly connected complexes on $D+k+2$ points. Note that
the number of terms in this sum is  independent of $n$. Any such
complex $K$ contains at least $\binom{k+2}{i+1} + D\binom{k+1}{i}$
faces of dimension $i$ for each $1 \le i \le k+1$; this counting is
presented in the proof of Lemma~8.1 in \cite{fowler:2019}. Hence, 
$$
\P(\sigma_{D+k+2}  \text{ is isomorphic to } K) \le
(D+k+2)!\prod_{i=q}^{k+1} p_i^{\binom{k+2}{i+1}+D\binom{k+1}{i}}, 
$$
and so, by \eqref{e:constraint.D}, 
\begin{align*}
\E(\beta_{k+1}-\tilde{\beta}_{k+1})  
&\le C n^{k+2+\tau_{k+1}(\balpha)-D(\psi_{k+1}(\balpha)-1)} \to 0, \ \
                                        \ n\to\infty. 
\end{align*}
Thus, \eqref{e:beta.k.and.beta.tilde} follows and, by Chebyshev's
inequality, the claim \eqref{e:betti.func.conv} is established once we
check that 
$$
\frac{\Var(\tilde{\beta}_{k+1})}{\Var(f_k)} \to 0, \ \ \ n\to\infty. 
$$
It suffices to show that for every $j=k+3, \dots, D+k+1$ and $r\ge 1,$
we have 
\begin{equation}  \label{e:main.goal}
\frac{\Var\Big(  \sum_{\sigma \subset [n], \, |\sigma|=j} \etasigjrko
  \Big)}{\Var(f_k)} \to 0, \ \ \ n\to\infty. 
\end{equation}
Simplifying the notation, we get 
\begin{align*}
\Var \Big(  \sum_{\sigma \subset [n], \, |\sigma|=j} \etasig \Big) &= \sum_{\substack{\sigma \subset [n],  \\ |\sigma|=j}} \sum_{\substack{\tau \subset [n],  \\ |\tau|=j}} \Big[ \E(\etasig \etatau) - \E(\etasig) \E(\etatau)  \Big]\\
&=\sum_{\ell=0}^j \sum_{\substack{\sigma \subset [n],  \\ |\sigma|=j}} \sum_{\substack{\tau \subset [n],  \\ |\tau|=j}} \Big[ \E(\etasig \etatau) - \E(\etasig) \E(\etatau)  \Big]\, \one \big\{ |\sigma\cap \tau|=\ell  \big\} \\
&=\sum_{\ell=0}^j \binom{n}{j} \binom{j}{\ell} \binom{n-j}{j-\ell}\Big[ \E(\etasig \etatau) - \E(\etasig) \E(\etatau)  \Big]\, \one \big\{ |\sigma\cap \tau|=\ell  \big\}.
\end{align*}
We consider six cases, depending on the value of $\ell:= |\sigma \cap \tau|$. 
\vspace{5pt}

\noindent $(I)$ $\ell \in \{ 0,\dots, q-2 \}$.

We claim that in this case the events underlying the indicator
functions $\eta_\sigma$ and $\eta_\tau$ are independent, so that the
corresponding terms have no contribution to the numerator in
\eqref{e:main.goal}. Indeed, the event underlying $\eta_\sigma$ can be
stated as saying that  the restriction of the complex to $\sigma$ is 
a $(k+1)$-strongly connected
subcomplex with   Betti number in dimension $k+1$  equal to $r$
and that no $(k+1)$-simplex carried by $\sigma$ has
$k+1$ common vertices, i.e., a common $k$-face, with a $(k+1)$-simplex not carried by $\sigma$. We
also have an analogous description of the event underlying
$\eta_\tau$. Stated this way, it is clear if a face $s_1$ plays a role
in the former event, and a face $s_2$  plays a role in the latter
event, then these faces have at most $q$ vertices in common and,
hence, the restrictions of the complex to these faces are
independent. 
\vspace{5pt}

\noindent $(I\hspace{-1.5pt}I)$ $\ell =q-1$.  

First, let
$$
\gamma_k  :=\prod_{i=q}^{k+1} p_i^{\binom{k+1}{i}}=n^{-\psi_{k+1}(\balpha)}
$$
denote the probability that a fixed $k$-face and a vertex not in that face form a $(k+1)$-simplex. 
For $j\in \{k+3,\dots, D+k+1\}$ and $r\ge1$, let $K$ denote a fixed $(k+1)$-strongly connected complex on $j$ vertices whose Betti number in dimension $k+1$ is equal to $r$. For $\sigma \subset [n]$ with $|\sigma|=j$, let $A_K$ be the event that the restriction of the complex to $\sigma$ is isomorphic to $K$, and define $q_K:=\P(A_K)$.  

We first claim that, for every $\sigma \subset [n]$ with $|\sigma|=j$, 
\begin{equation}  \label{e:prob.etasig}
\E(\etasig)=\sum_{K: |K|=j} q_K (1-s_K \gamma_k +u_K)^{n-j}, 
\end{equation}
where the sum is taken over all $(k+1)$-strongly connected complexes, up to an isomorphism class, such that the Betti number in dimension $k+1$ is equal to $r$. Moreover, $s_K$ is the number of $k$-faces in $K$, and $u_K=\mathcal O(\gamma_k)$ as functions of $n,$ i.e., there exists $C>0$ such that $u_K/\gamma_k < C$ for all $n\geq 1$ and all $K$. Note that $q_K, \gamma_k$, and $u_K$ depend on $n$, whereas $s_K$ is independent of $n$. For the proof of \eqref{e:prob.etasig}, write
$$
\E(\etasig) = \sum_{K: |K|=j} q_K \P(\sigma \text{ is maximal } | A_K). 
$$
Let us fix a vertex $v\in\sigma^c$. By the inclusion-exclusion formula, the probability of forming at least one $(k+1)$-simplex between $v$ and a $k$-face in $\sigma$, can be written as $s_K \gamma_k -u_K$. 
The largest term in $u_K$ corresponds to $v$ forming two $(k+1)$-simplices with $k$-faces $f_1$ and $f_2$ respectively, such that $\dim(f_1 \cap f_2)=k-1$. Therefore, the largest term in $u_K$ is of the order $\gamma_k^2 \gamma_{k-1}^{-1} = \mathcal O(\gamma_k)$. 
Since there are $n-j$ vertices in $\sigma^c$, we have 
$$
\P(\sigma \text{ is maximal } | A_K)=(1-s_K \gamma_k +u_K)^{n-j}, 
$$
and \eqref{e:prob.etasig} follows as required. 

Next, let $K$, $K'$ be fixed $(k+1)$-strongly connected complexes on $j$ vertices with Betti number in dimension $k+1$ equal to $r$. Denote by $A_{K, K'}$ the event that the restriction of the complex to $\sigma$ and that to $\tau$ are isomorphic to $K$ and $K'$, respectively. 
It then follows from \eqref{e:prob.etasig} 
that 
\begin{align}
&\big[ \E(\etasig \etatau) - \E(\etasig) \E(\etatau)  \big]\, \one \big\{ |\sigma\cap \tau|=q-1  \big\}\label{e:case1.goal}\\
&= \sum_{K: |K|=j} \sum_{K': |K'|=j} \Big[ \P(\sigma \text{ and } \tau \text{ are maximal }| A_{K, K'}) \P(A_{K, K'}) \notag\\
&\qquad \qquad \qquad \qquad \quad - q_Kq_{K'} (1-s_K\gamma_k +u_K)^{n-j}(1-s_{K'}\gamma_k +u_{K'})^{n-j} \Big]\, \one \{ |\sigma \cap \tau|=q-1 \}, \notag
\end{align}
where the sums are again taken over all $(k+1)$-strongly connected complexes whose Betti numbers in dimension $k+1$ are equal to $r$, and 
$s_{K'}$, $u_{K'}$ are defined analogously to those for $K$. Since $|\sigma\cap \tau|=q-1$ and all the $(q-2)$-faces exist with probability one, we have $\P(A_{K, K'})=q_K q_{K'}$. 
For every $v\in (\sigma\cup \tau)^c$, let $B_v$ be the event that $v$ forms a $(k+1)$-simplex with a $k$-face in $\sigma \cup \tau$. Further, let $D_1$ denote the event that at least one $(k+1)$-simplex exists between a $k$-face in $\sigma$ and a point in $\tau\setminus (\sigma \cap \tau)$, and $D_2$ is an event obtained by switching the role of $\sigma$ and $\tau$. Then, by independence we see that 
\begin{align}
 \P(\sigma \text{ and } \tau \text{ are maximal }| A_{K, K'}) &= \P \bigg( \Big(\bigcap_{v \in (\sigma\cup \tau)^c} B_v^c\Big) \cap D_1^c \cap D_2^c\,  \Big| \, A_{K, K'}  \bigg) \label{e:conditional.maximal} \\
 &= \prod_{v\in (\sigma \cup \tau)^c} \big( 1-\P(B_v | A_{K, K'}) \big) \P(D_1^c \cap D_2^c | A_{K, K'}). \notag
\end{align}
By the inclusion-exclusion formula, we have 
\begin{align}
\P(B_v | A_{K, K'}) = (s_K+s_{K'})\gamma_k - u_K - u_{K'} -s_Ks_{K'} \gamma_k^2 + s_K\gamma_k u_{K'}+ s_{K'}\gamma_k u_K - u_Ku_{K'} =: a_{K, K'}. \label{e:def.aKK'}
\end{align}
Indeed, the probabilities that $v$ forms $(k+1)$-simplices with multiple $k$-faces in $\sigma$ are grouped into $u_K$, while the probabilities that $v$ forms $(k+1)$-simplices with multiple $k$-faces in $\tau$ are grouped into $u_{K'}$. Moreover, the probabilities that $v$ forms $(k+1)$-simplices with both $k$-faces in $\sigma$ and those in $\tau$, are grouped into one of the last four terms in \eqref{e:def.aKK'}. Above, we have also exploited the fact that the events concerning $v$ forming  $(k + 1)$-simplices with $k$-faces in $\sigma$ are independent from events concerning $v$ forming $(k + 1)$-simplices with $k$-faces in $\tau.$


Noting that there are $n-2j+q-1$ points in $(\sigma \cup \tau)^c$, the right hand side of \eqref{e:case1.goal} is equal to
\begin{equation}  \label{e:bound.case1.goal}
\sum_{K: |K|=j} \sum_{K': |K'|=j}  q_Kq_{K'} (1-a_{K, K'})^{n-2j+q-1} \big[ \P(D_1^c\cap D_2^c | A_{K, K'}) - (1-a_{K, K'})^{j-q+1} \big]. 
\end{equation}
By the binomial expansion, it is easy to see that 
\begin{align*}
&(1-a_{K,K'})^{n-2j+q-1} = \big( 1-\mathcal O(\gamma_k) \big)^{n-2j+q-1} = \mathcal O(1), \\
&(1-a_{K, K'})^{j-q+1} = 1-(j-q+1) a_{K, K'} + \mathcal O(\gamma_k^2), 
\end{align*}
and, further, 
\begin{align*}
\P(D_1|A_{K,K'}) &= 1-(1-s_K\gamma_k + u_K)^{j-q+1} = (j-q+1)(s_K\gamma_k-u_K) - \mathcal O(\gamma_k^2), \\
\P(D_2|A_{K,K'}) &= 1-(1-s_{K'}\gamma_k + u_{K'})^{j-q+1} = (j-q+1)(s_{K'}\gamma_k-u_{K'}) - \mathcal O(\gamma_k^2). 
\end{align*}

Suppose now that 
\begin{equation}  \label{e:q.pts.concentrated}
\text{there exist two } k\text{-faces } f_1\subset \sigma \text{ and } f_2 \subset \tau \text{ such that } |f_1\cap f_2|=q-1. 
\end{equation}
Under \eqref{e:q.pts.concentrated}, we claim that 
$$
\P(D_1\cap D_2|A_{K,K'}) = \mathcal O(\gamma_k^2p_q^{-1}). 
$$
Indeed, the largest term in the right hand side corresponds to the case in which a vertex in $f_1\setminus (f_1\cap f_2)$ forms a $(k+1)$-simplex with $f_2$, and a vertex in $f_2\setminus (f_1\cap f_2)$ forms a $(k+1)$-simplex with $f_1$. 
Because of a double-count of a $q$-face consisting of the vertices in $f_1 \cap f_2$ and the two selected vertices, the largest rate is of order $\gamma_k^2 p_q^{-1}$. By combining all these results, it is now straightforward to get that 
\begin{equation}  \label{e:main.term.D}
\P(D_1^c\cap D_2^c | A_{K, K'}) - (1-a_{K, K'})^{j-q+1} =\mathcal O(\gamma_k^2p_q^{-1}).
\end{equation}

If \eqref{e:q.pts.concentrated} does not hold, the same analysis gives the behavior as in \eqref{e:main.term.D}, but with a smaller correction term; $\mathcal O(\gamma_k^2)$ instead of $\mathcal O(\gamma_k^2 p_q^{-1})$. From all of these results, \eqref{e:bound.case1.goal} can be written as  
$$
C \sum_{\substack{K: |K|=j \\ \eqref{e:q.pts.concentrated} \text{holds}}} \sum_{K': |K'|=j}  q_Kq_{K'} \mathcal O(\gamma_k^2p_q^{-1}), 
$$
and, thus, 
\begin{align*}
&\binom{n}{j}\binom{j}{q-1}\binom{n-j}{j-q+1} \big[ \E(\etasig \etatau) - \E(\etasig) \E(\etatau)  \big]\, \one \big\{ |\sigma\cap \tau|=q-1  \big\} \\
&\le C \sum_{\substack{K: |K|=j \\ \eqref{e:q.pts.concentrated} \text{holds}}} \sum_{K': |K'|=j}  n^{2j-q+1}q_Kq_{K'} \mathcal O(\gamma_k^2p_q^{-1}).
\end{align*}
By Lemma \ref{l:formula.fowler}, 
\begin{align*}
n^j q_K &\le  C n^j  \prod_{i=q}^{k+1} p_i^{\binom{k+3}{i+1}} \Big( \prod_{i=q}^{k+1} p_i^{\binom{k+1}{i}} \Big)^{j-k-3} =  C n^{\tau_{k+2}(\balpha)+\al_{k+2}}\big(n^{1-\psi_{k+1}(\balpha)}\big)^{j-k-3}. 
\end{align*}
Since $\psi_{k+1}(\balpha)>1$ and $j\ge k+3$, we get $\big( n^{1-\psi_{k+1}(\balpha)} \big)^{j-k-3}\le 1$, and hence, 
\begin{equation}  \label{e:bound.nj.qK}
n^{2j}q_Kq_{K'} \le  C n^{2(\tau_{k+2}(\balpha)+\al_{k+2})}. 
\end{equation}
It now remains to check that 
$$
\frac{n^{2(\tau_{k+2}(\balpha)+\al_{k+2})}\mathcal O(n^{-q+1}\gamma_k^2p_q^{-1})}{\Var (f_k)} \to 0 \ \ \text{as } n\to\infty. 
$$
But this actually follows, since by Proposition \ref{p:moment.face.count} and \eqref{e:simple.lemma}, the expression on the left hand side is bounded by 
$$
Cn^{2(\tau_{k+1}(\balpha)-\tau_k(\balpha))} \mathcal O(n^{\tau_q(\balpha)-q+1}\gamma_k^2p_q^{-1}) = Co(1)\mathcal O(n^{2(1-\psi_{k+1}(\balpha))})\to 0, \ \ \ n\to\infty. 
$$
\vspace{5pt}

\noindent $(I\hspace{-1.5pt}I\hspace{-1.5pt}I)$ $\ell=q$. 

The case $\ell=q$ is similar but easier. 
Using the same notation as in Case ($I\hspace{-1.5pt}I$), we once again consider
\begin{align}
&\big[ \E(\etasig \etatau) - \E(\etasig) \E(\etatau)  \big]\, \one \big\{ |\sigma\cap \tau|=q  \big\} \label{e:case2.goal}  \\
&= \sum_{K: |K|=j} \sum_{K': |K'|=j} \Big[ \P(\sigma \text{ and } \tau \text{ are maximal }| A_{K, K'}) \P(A_{K, K'}) \notag \\
&\qquad \qquad \qquad \qquad \quad - q_Kq_{K'} (1-s_K\gamma_k +u_K)^{n-j}(1-s_{K'}\gamma_k +u_{K'})^{n-j} \Big]\, \one \{ |\sigma \cap \tau|=q \}. \notag
\end{align}
Since $|\sigma \cap \tau|=q$ and all the $(q-1)$-faces exist with probability one, we still get $\P(A_{K,K'})=q_Kq_{K'}$. By the same reasoning as before, we only consider the situation that 
\begin{equation}  \label{e:q.pts.concentrated1}
\text{there exist two } k\text{-faces } f_1\subset \sigma \text{ and } f_2 \subset \tau \text{ such that } |f_1\cap f_2|=q. 
\end{equation}
Under this assumption, for each $v\in(\sigma\cup \tau)^c$, the inclusion-exclusion formula gives that 
$$
\P(B_v|A_{K,K'}) = (s_K+s_{K'})\gamma_k -u_K-u_{K'} -\mathcal O(\gamma_k^2p_q^{-1}), 
$$
The largest term in the big-$\mathcal O$ expression is associated with the case in which $v$ forms two $(k+1)$-simplices with $f_1$ and $f_2,$ respectively. By \eqref{e:conditional.maximal}, we see that 
\begin{align*}
&\P(\sigma \text{ and } \tau \text{ are maximal }| A_{K, K'}) \le \prod_{v\in(\sigma\cup \tau)^c} \big( 1-\P(B_v|A_{K,K'}) \big) \\
&\qquad =\big( 1-  (s_K+s_{K'})\gamma_k +u_K+u_{K'} +\mathcal O(\gamma_k^2p_q^{-1})\big)^{n-2j+q} \\
&\qquad =  \big( 1-  (s_K+s_{K'})\gamma_k +u_K+u_{K'} \big)^n (1 + \mathcal{O}(\gamma_k))^{-(2j - q)}(1 + \mathcal{O}(\gamma_k^2 p_q^{-1}))^n\\
&\qquad = \big( 1-  (s_K+s_{K'})\gamma_k +u_K+u_{K'} \big)^n \big(1+\mathcal O(\gamma_kp_q^{-1})  \big).
\end{align*}
Here, we have made use of the following facts:  $\gamma_k^2 p_q^{-1} = \mathcal O(\gamma_k),$ $(1 + O(\gamma_k))^{-(2j - q)} = 1 + O(\gamma_k),$ 
and $(1 + \mathcal O(\gamma_k^2 p_q^{-1}))^{n} = 1 + \mathcal{O}(n \gamma_k^2 p_q^{-1}) = 1+ \mathcal{O}(\gamma_k p_q^{-1}).$

Similarly, we derive that 
\begin{align*}
 (1-s_K\gamma_k +u_K)^{n-j}(1-s_{K'}\gamma_k +u_{K'})^{n-j} &= \big( 1-  (s_K+s_{K'})\gamma_k +u_K+u_{K'} +\mathcal O(\gamma_k^2) \big)^{n-j} \\
 &= \big( 1-  (s_K+s_{K'})\gamma_k +u_K+u_{K'}  \big)^n \big( 1+\mathcal O(\gamma_k) \big).
\end{align*}
Putting all these results together, along with the binomial expansion $\big( 1-  (s_K+s_{K'})\gamma_k +u_K+u_{K'}  \big)^n = \mathcal O(1)$ as $n\to\infty$, 
we can conclude that 
$$
\binom{n}{j} \binom{j}{q} \binom{n-j}{j-q} \big[ \E(\etasig \etatau) - \E(\etasig) \E(\etatau)  \big]\, \one \big\{ |\sigma\cap \tau|=q  \big\} \le C \hspace{-10pt} \sum_{\substack{K: |K|=j \\ \eqref{e:q.pts.concentrated1} \text{ holds}}} \sum_{K': |K'|=j}  n^{2j-q} q_Kq_{K'} \mathcal O(\gamma_k p_q^{-1}). 
$$
Using Lemma \ref{l:formula.fowler} as in \eqref{e:bound.nj.qK}, it follows that the right hand side above can be bounded  by $C n^{2(\tau_{k+2}(\balpha)+\al_{k+2})} \mathcal O(n^{-q}\gamma_k p_q^{-1})$. 
Finally, Proposition \ref{p:moment.face.count} and \eqref{e:simple.lemma} help to conclude that 
\begin{align*}
\frac{n^{2(\tau_{k+2}(\balpha)+\al_{k+2})}\mathcal O(n^{-q}\gamma_k p_q^{-1})}{\Var(f_k)} &\le C n^{2(\tau_{k+1}(\balpha)-\tau_k(\balpha))}\mathcal O(n^{\tau_q(\balpha)-q-\psi_{k+1}(\balpha)+\al_q}) \\
&=C o(1)\mathcal O(n^{1-\psi_{k+1}(\balpha)}) \to 0, \ \ \ n\to\infty. 
\end{align*}
\vspace{5pt}

\noindent $(IV)$ $\ell\in \{q+1,\dots,k+2\}$. 

Note first that 
\begin{align*}
&\binom{n}{j}\binom{j}{\ell}\binom{n-j}{j-\ell}\Big[
    \E(\etasig\etatau) - \E(\etasig)\E(\etatau)\Big]\,  \one \big\{ |\sigma \cap \tau|=\ell \big\}  
\le  n^{2j-\ell} \E(\etasig\etatau)\,   \one \big\{ |\sigma \cap \tau|=\ell \big\}  \\
&\qquad \qquad \qquad \le  n^{2j-\ell}\sum_{K:|K|=j}\sum_{K':|K'|=j} \P(A_{K, K'}) \,  \one \big\{ |\sigma \cap \tau|=\ell \big\}.   
\end{align*}
where $A_{K, K'}$ is as in Case ($I\hspace{-1.5pt}I$).
Since there are finitely many  isomorphism classes of $(k+1)$-strongly connected complexes on $j$ vertices, we only have to show that for
all such $K, K'$ with $|\sigma \cap \tau|=\ell$,
\begin{equation}  \label{e:transfer.to.Case4}
\big( \Var(f_k) \big)^{-1} n^{2j-\ell}  \P(A_{K, K'})
\to 0, \ \ n\to\infty.
\end{equation}
By  Lemma \ref{l:formula.fowler},  
\begin{align*}
&  \P(A_{K, K'})
\le  C \bigg[ \prod_{i=q}^{k+1} p_i^{\binom{k+3}{i+1}} \Big(
                \prod_{i=q}^{k+1} p_i^{\binom{k+1}{i}} \Big)^{j-k-3}
                \bigg]^2\times
                \prod_{i=q}^{k+1}p_i^{-\binom{\ell}{i+1}},    
\end{align*}
with the last factor accounting for the faces on the vertices  common
to $\sigma$ and $\tau$.  We conclude  that 
\begin{align*}
&n^{2j-\ell}  \P(A_{K, K'})
\le C n^{2j-\ell} \prod_{i=q}^{k+1} p_i^{2\binom{k+3}{i+1}-\binom{\ell}{i+1}} \bigg( \prod_{i=q}^{k+1} p_i^{\binom{k+1}{i}} \bigg)^{2(j-k-3)} \\
=& Cn^{2(\tau_{k+2}(\balpha)+\al_{k+2})-\tau_{\ell-1}(\balpha)} \big( n^{1-\psi_{k+1}(\balpha)} \big)^{2(j-k-3)} 
\le  Cn^{2(\tau_{k+2}(\balpha)+\al_{k+2})-\tau_{\ell-1}(\balpha)}. 
\end{align*}
By  Proposition \ref{p:moment.face.count} and \eqref{e:simple.lemma}, 
$$
\frac{n^{2(\tau_{k+2}(\balpha)+\al_{k+2})-\tau_{\ell-1}(\balpha)}}{\Var(f_k)}
\le C
n^{2(\tau_{k+1}(\balpha)-\tau_k(\balpha))+(\tau_q(\balpha)-\tau_{\ell-1}(\balpha))}\to 0
$$
because the exponent is clearly negative if $\ell\in\{q+1,\dots,k+1 \}$, and it is still true in the case
$\ell=k+2$, because 
$$
2\big(\tau_{k+1}(\balpha)-\tau_k(\balpha)\big)+\big(\tau_q(\balpha)-\tau_{\ell-1}(\balpha)\big) = \big(\tau_{k+1}(\balpha)-\tau_k(\balpha)\big) + \big(\tau_q(\balpha)-\tau_k(\balpha)\big) < 0.
$$
\vspace{5pt}

\noindent $(V)$ $\ell\in \{ k+3,\dots,j-1 \}$. 

It is still sufficient to prove \eqref{e:transfer.to.Case4}, which we
presently do. We note that
\begin{align*}
 \P(A_{K, K'}) = &\P (\text{the complex
                          restricted to $\sigma$   is isomorphic to $K$ }) \\
  &\P (\text{the complex
   restricted to $\tau$   is isomorphic to $K'$ })
 D(K,K'),
\end{align*}
where $D(K,K')$ is the correction term, resulting from the fact that some of the faces in the restriction of the complex to $\sigma\cap\tau$ are used in both $K$ and $K'$. Hence, for each fixed
$\ell$, we 
obtain an upper bound on $\P(A_{K, K'})$ 
by considering the worst case scenario (from the perspective of showing \eqref{e:transfer.to.Case4}).

To see how it works, consider the case $\ell=k+3$. Clearly, the worst case scenario is when both $K$ and $K'$ have the least number of $i$-faces for $q \leq i \leq k + 1;$ further, in the complex restricted to $\sigma \cap \tau,$ there is a maximum overlapping of faces. However, since $K$ and $K'$ are $(k + 1)$-strongly connected, even in this worst case scenario, the complex restricted to the $k + 3$ vertices in $\sigma \cap \tau$ should have at least two $(k + 1)$-faces; of course, these two may have a common shared $k$-face.

 
Hence, 
$$
D(K,K')=\prod_{i=q}^{k+1}p_i^{-\binom{k+2}{i+1}}
\prod_{i=q}^{k+1} p_i^{-\binom{k+1}{i}};
$$
so, by Lemma \ref{l:formula.fowler}, we have 
\begin{align*}
&n^{2j-(k+3)} 
 \P(A_{K, K'})
\le  Cn^{2j-(k+3)}\bigg[
                \prod_{i=q}^{k+1} p_i^{\binom{k+3}{i+1}}
                \Big( \prod_{i=q}^{k+1} p_i^{\binom{k+1}{i}} \Big)^{j-k-3} \bigg]^2 \prod_{i=q}^{k+1}p_i^{-\binom{k+2}{i+1}} \prod_{i=q}^{k+1} p_i^{-\binom{k+1}{i}}.
\end{align*}

Suppose next that $\ell=k+4$. In the worst case scenario now, 
the restriction of the complex to $k+3$   (out of the
$k+4$)  common points of the intersection should have the same setup as in the previous case,
while 
the last $(k+4)$th common point should form a
$(k+1)$-simplex with one of the two $(k+1)$-simplices constructed before.
Once again, this is the minimal requirement since both $K$ and $K'$ are $(k+1)$-strongly connected. Hence,
$$
D_{\sigma,\tau}(K,K')=\prod_{i=q}^{k+1}p_i^{-\binom{k+2}{i+1}}
\left( \prod_{i=q}^{k+1} p_i^{-\binom{k+1}{i}}\right)^2;
$$
so, by Lemma \ref{l:formula.fowler}
\begin{align*}
&n^{2j-(k+4)}
 \P(A_{K, K'})
\le C n^{2j-(k+4)}\bigg[
                \prod_{i=q}^{k+1} p_i^{\binom{k+3}{i+1}}
                \Big( \prod_{i=q}^{k+1} p_i^{\binom{k+1}{i}} \Big)^{j-k-3} \bigg]^2 \prod_{i=q}^{k+1}p_i^{-\binom{k+2}{i+1}} \left(\prod_{i=q}^{k+1} p_i^{-\binom{k+1}{i}}\right)^2.
\end{align*}
Proceeding in the same manner for any  $\ell \in \{ k+3,\dots,j-1\}$,
we see that
\begin{align*}
&n^{2j-\ell}  
 \P(A_{K, K'})
\le   Cn^{2j-\ell}\bigg[ \prod_{i=q}^{k+1} p_i^{\binom{k+3}{i+1}} \Big( \prod_{i=q}^{k+1} p_i^{\binom{k+1}{i}} \Big)^{j-k-3} \bigg]^2 \prod_{i=q}^{k+1}p_i^{-\binom{k+2}{i+1}} \Big( \prod_{i=q}^{k+1} p_i^{-\binom{k+1}{i}} \Big)^{\ell-(k+2)}. 
\end{align*}

Therefore, as before,
\begin{align*}
&n^{2j-\ell}\bigg[ \prod_{i=q}^{k+1} p_i^{\binom{k+3}{i+1}} \Big( \prod_{i=q}^{k+1} p_i^{\binom{k+1}{i}} \Big)^{j-k-3} \bigg]^2 \prod_{i=q}^{k+1}p_i^{-\binom{k+2}{i+1}} \Big( \prod_{i=q}^{k+1} p_i^{-\binom{k+1}{i}} \Big)^{\ell-(k+2)} \\
&= n^{2(\tau_{k+2}(\balpha)+\al_{k+2}) - \tau_{k+1}(\balpha)} \big( n^{1-\psi_{k+1}(\balpha)} \big)^{2j-k-\ell-4} \\
&\le n^{2(\tau_{k+2}(\balpha)+\al_{k+2}) - \tau_{k+1}(\balpha)},
\end{align*}
which is the same bound as that for $\ell=k+2$ in the previous case. Thus, we get \eqref{e:transfer.to.Case4}, as desired.
\vspace{5pt}

\noindent $(VI)$ $\ell=j$.

We again prove \eqref{e:transfer.to.Case4}, this time only with
$K=K'$. Now, by Lemma \ref{l:formula.fowler}, 
\begin{align*}
n^{j} 
 \P(A_{K, K'})
&\le  C n^j \prod_{i=q}^{k+1} p_i^{\binom{k+3}{i+1}} \Big(
                              \prod_{i=q}^{k+1} p_i^{\binom{k+1}{i}}
                              \Big)^{j-k-3} \\
&=  C n^{\tau_{k+2}(\balpha) + \al_{k+2}} \big( n^{1-\psi_{k+1}(\balpha)}
                                                 \big)^{j-k-3} \le
                                                 n^{\tau_{k+2}(\balpha)+\al_{k+2}}, 
\end{align*}
and,  by Proposition \ref{p:moment.face.count} and  \eqref{e:simple.lemma},
$$
\frac{n^{\tau_{k+2}(\balpha)+\al_{k+2}}}{\Var(f_k)} \le C n^{(\tau_{k+1}(\balpha)-\tau_k(\balpha)) + (\tau_q(\balpha)-\tau_k(\balpha))} \to 0,\ \ n\to\infty.
$$

\medskip

This completes the proof of \eqref{e:main.goal} and, hence, of
\eqref{e:betti.func.conv} in Theorem
\ref{t:clt.topological.invariants}. 

Finally, assuming \eqref{e:cond.regularity} and
\eqref{e:sharp.drop.k+1}, we 
establish tightness in the Skorohod $J_1$-topology. 
First of all, we already proved that under these assumptions, the convergence in \eqref{e:Euler.characteristic.func.conv} holds in the sense of weak convergence in the $J_1$-topology on $D[0,\infty)$. Fixing $T>0$ and choosing $m$ so large that $T/m \le a/4$ with $a$ defined in \eqref{e:assumption.Gi}, we again consider a static multi-parameter simplicial complex $X([n], \bp^{(1)})$ and the corresponding $j$-face counts $f_j^{(1)}$, that were used for the proof of \eqref{e:Euler.as.conv}.  
By Proposition \ref{p:vanishing.lower.order.betti} in the
Appendix, 
all we have to do is to show that 
$$
\left(\frac{\sum_{j=k+1}^{n-1} (-1)^j \beta_j(t) -\E \Big(
  \sum_{j=k+1}^{n-1} (-1)^j \beta_j \Big)}{\sqrt{\Var(f_k)}} , \, 0 \le t \le \frac{T}{m}
  \right)
\to  {\bf 0}  
$$
in probability in the $J_1$-topology. This will follow once we show that  for every $\epsilon>0$, 
\begin{equation}  \label{e:sup.upper.cond.tightness}
\P \bigg( \sup_{0\le t \le T/m} \left|  \sum_{j=k+1}^{n-1} (-1)^j \beta_j(t) - \E\Big(\sum_{j=k+1}^{n-1} (-1)^j \beta_j  \Big) \right| > \epsilon \sqrt{\Var(f_k)}  \bigg) \to 0, \ \ \ n\to\infty. 
\end{equation}
To this end, observe that by \eqref{e:sharp.drop.k+1}, for any 
$j \ge k+1$, we have 
\begin{equation}  \label{e:under4.5}
E(f_j) =\mathcal O(n^{\tau_{k+1}(\balpha)}) = o \big( n^{\tau_k(\balpha)-\tau_q(\balpha)/2} \big) = o \big( \sqrt{\Var(f_k)} \big), \ \ \ n\to\infty. 
\end{equation}
Proceeding as in \eqref{e:Markov.chi2}, while using $M(\tilde \balpha)$ defined in \eqref{e:M.alpha} and \eqref{e:alpha.tilde.rate}, we can bound the left hand side of
\eqref{e:sup.upper.cond.tightness} by 
\begin{align*}
&\frac{2}{\epsilon \sqrt{\Var(f_k)}}\, \sum_{j=k+1}^{n-1} \E\big[\sup_{0\le t \le T/m}f_j(t)\big]  \le \frac{2}{\epsilon \sqrt{\Var(f_k)}}\, \sum_{j=k+1}^{n-1} \E(f_j^{(1)}) \\
&\le \frac{2}{\epsilon} \sum_{j=k+1}^{M(\tilde \balpha)-1} \frac{\prod_{i=q}^j 2^{\binom{j+1}{i+1}}\E(f_j)}{\sqrt{\Var(f_k)}} +\frac{2}{\epsilon} \sum_{j=M(\tilde \balpha)}^\infty \E(f_j^{(1)}). 
\end{align*}
The last term converges to $0$ as $n\to\infty$ due to \eqref{e:under4.5} and Corollary \ref{cor:neg.tau}. 
\end{proof}


\section{Appendix} 
\subsection{Analysis of the Betti numbers in lower dimensions}   

We begin with introducing additional notions of connectivity. 
Given a simplicial complex $X$ and an $\ell$-dimensional simplex
$\sigma$ in $X$, let the simplicial complex $\lk_X(\sigma) := \{\tau \in X: \sigma \cap \tau = \emptyset, \sigma \cup \tau \in X\}$ denote the \textit{link} of $\sigma$ in $X$.  
In other words, $\lk_X(\sigma)$ denotes the subcomplex of $X$ consisting of all simplices whose
vertex support is disjoint from that of $\sigma$ but, together with
$\sigma$,  they form a simplex in $X$.  
If $X$ is pure $\ell$-dimensional and $\sigma$ is
$(\ell-2)$-dimensional for some $\ell \geq 2$, then $\text{lk}_X(\sigma)$ necessarily is a one-dimensional simplicial complex.
We say that an $(\ell-1)$-face in $X$ is \textit{free} if it is not contained in any of the $\ell$-faces in $X$. 
Given a graph $G$, we denote by $\lambda_2(G)$ the second smallest
eigenvalue of the normalized graph Laplacian of $G$. We will use the
\textit{cohomology vanishing theorem}  of
\cite{ballmann:swiatkowski:1997}: if $X$ is a finite pure
$\ell$-dimensional simplicial complex such that for every
$(\ell-2)$-simplex $\sigma\in X$, the link $\text{lk}_X(\sigma)$ is connected and has spectral gap $\lambda_2\bigl(
\text{lk}_X(\sigma)\bigr)>1-1/\ell$, then $H^{\ell-1}(X; \bbq)=0$. In
particular, $\beta_{\ell-1}(X)=0$.

\begin{proposition}  \label{p:vanishing.lower.order.betti}
Under the assumptions of Theorem \ref{t:clt.topological.invariants}, 
$$
\left( \frac{\beta_j(t)-\E(\beta_j)}{\sqrt{\Var(f_k)}}, \, t\geq
  0\right)
\to {\bf 0} \  \text{ in } D[0,\infty)
$$
in probability as $n\to\infty$ for all $j=0,1,\dots, k-1$, where $\bf 0$ is the
constant zero process.
\end{proposition}
\begin{proof}
If $k=1$ the claim is trivial, so assume that $k\geq 2$. 
We consider $j=k-1$ only; smaller dimensions can be treated in a
similar way. 
Proposition \ref{p:vanishing.lower.order.betti} will be established by combining a series of lemmas provided below. 
Let $F_{j}(t)$ be the number of free $j$-faces of $X([n],\bp;
t)$, and $X_k(t)$  the $k$-skeleton of $X([n],\bp; t)$. For a
$(k-2)$-face $\sigma$ in $X_k(t)$, write $L_\sigma(t) :=
|\text{lk}_{X_k(t)}(\sigma)|$, i.e., the number of vertices
in the link of
$\sigma$ in $X_k(t)$. We set $F_{j}:= F_{j}(0)$, $X_k:=X_k(0)$, and $L_\sigma:= L_\sigma(0)$. 

Consider the delayed renewal sequences defined in
\eqref{e:renewal.seq} corresponding to the 
stationary renewal processes $\big( \Delta_{i,A}, \, q\le i
\le k, \, A\in \W_i \big)$. 
Enumerating the different arrival times, 
we denote the resulting sequence by $ \eta_1
\leq \eta_2 \leq \cdots$, and set $\eta_0=0$.  For $0 < T < \infty,$ we
denote by $N(T)$ the number of these points in the interval
$[0,T]$. Clearly, $\E \big(N(T)\big) = \mathcal O(n^{k + 1})$ for every such $T$. 
\begin{lemma}   \label{l:exp.number.free.face}
For each $0\le j\leq k-1$,
$$
\E(F_{j}) = o (e^{-n^\epsilon}), \ \ n\to\infty,
$$
for some $\epsilon>0$. 
\end{lemma}
\begin{proof}
A simple calculation shows that 
$$
\E(F_j) \le n^{\tau_{j}(\balpha)}\left( 1-n^{-\psi_{j+1}(\balpha)}\right)^{n-j-1}.
$$
If $\psi_{j+1}(\balpha)=0$, the claim is trivial. Otherwise, 
$$
\E(F_{j}) \le C n^{\tau_{j}(\balpha)} e^{-n^{1-\psi_{j+1}(\balpha)}}. 
$$
Since $\psi_{j+1}(\balpha)\leq \psi_k(\balpha) < 1$, the
result follows. 
\end{proof}
\begin{lemma}  \label{l:pureness}
$$
\P(X_k \text{ is pure}) = 1-o(e^{-n^\epsilon}), \ \ n\to\infty, 
$$
for some $\epsilon>0$. 
\end{lemma}
\begin{proof} By Lemma \ref{l:exp.number.free.face}, 
\begin{align*}
\P(X_k \text{ is pure}) &= \P(F_j=0, \ \ j=0,\dots,k-1) \\
&\ge 1-\sum_{j=0}^{k-1} \P(F_j \ge 1) \ge 1-\sum_{j=0}^{k-1} \E(F_j) = 1-o(e^{-n^\epsilon}). 
\end{align*}
\end{proof}
\begin{lemma}  \label{l:log.L.sigma}
Fix $\delta>0$. For a $(k-2)$-face $\sigma$ of $X_k$, 
$$
\P \left( \frac{(1+\delta)\log L_\sigma}{L_\sigma} > p_1  \right) = o(e^{-n^\epsilon}), \ \ n\to\infty,
$$
for some $\epsilon>0$. 
\end{lemma}
\begin{proof}
Note that
$$
1-\psi_{k-1}(\balpha) >\psi_k(\balpha)-\psi_{k-1}(\balpha)\geq
\alpha_1, 
$$
and $(1+\delta)x^{-1}\log x$ is decreasing for $x\ge e$.
Therefore, if $L_\sigma\geq n^{1-\psi_{k-1}(\balpha)}/2$, then 
$$
\frac{(1+\delta)\log L_\sigma}{L_\sigma} \leq
\frac{(1+\delta)\log \bigl(n^{1-\psi_{k-1}(\balpha)}/2\bigr)}
{n^{1-\psi_{k-1}(\balpha)}/2}<n^{-\alpha_1}=p_1
$$
for large $n$. Hence, for large $n$,
$$
\P \left( \frac{(1+\delta)\log L_\sigma}{L_\sigma} > p_1  \right)
\leq \P\left( L_\sigma<\frac{n^{1-\psi_{k-1}(\balpha)}}{2}\right), 
$$
and the claim follows from the basic properties of the binomial
distribution because $L_\sigma$ has a binomial distribution with
parameters $n-k+1$ and $n^{-\psi_{k-1}(\balpha)}$; see, e.g., Lemma 4.2 in \cite{fowler:2019}.
\end{proof}
\begin{lemma}  \label{l:betti.k-1}
For every $0<T<\infty$, 
$$
\P\Big(\sup_{0\le t \le T}\beta_{k-1}(t) \neq 0 \Big) = \mathcal O(n^{-k-1}), \ \ n\to\infty. 
$$
\end{lemma}
\begin{proof}
By the cohomology vanishing theorem, 
\begin{align*}
&\P \big( \sup_{0\le t \le T} \beta_{k-1}(t) =0 \big) = \P\Big(  \sup_{0 \le t \le T} \beta_{k-1}\big( X_k(t) \big) =0 \Big) \\
&=\P \Big(\beta_{k-1}\big( X_k(\eta_\ell) \big)=0 \text{ for } \ell =    0,1,\dots, \ N(T)  \Big) \\
&\ge \P \Big(\beta_{k-1}\big( X_k(\eta_\ell) \big)=0  \text{ for } \ell =   0,1,\dots, n^{2k+2}, \,  N(T) \le n^{2k+2} \Big) \\ 
&\ge \P \Big(  \bigcap_{\ell=0}^{n^{2k+2}} \Big( \Big\{  \lambda_2 \big( \text{lk}_{X_k(\eta_\ell)}(\sigma) \big)> 1-\frac{1}{k}  \text{ and } \text{lk}_{X_k(\eta_\ell)}(\sigma) \text{ is connected} \\ 
&\qquad \qquad \text{for every }(k-2)\text{-face } \sigma \text{ in } X_k(\eta_\ell) \Big\}  \cap \Big\{ X_k(\eta_\ell) \text{ is pure}
   \Big\}\Big) \cap \Big\{  N(T) \le n^{2k+2} \Big\} \Big) \\
      &\ge 1-\sum_{\ell=0}^{n^{2k+2}} \binom{n}{k-1}
        \P \left(    \lambda_2
        \big(  \text{lk}_{X_k(\eta_\ell)}(\sigma_0) \big) \le
        1-\frac{1}{k}  \text{ or } \text{lk}_{X_k(\eta_\ell)}(\sigma_0) \text{ is disconnected} \right) \\
  &    \quad \ \ -\sum_{\ell=0}^{n^{2k+2}}\P \bigl( X_k(\eta_\ell) \text{ is
    not pure}  \big)  - \P\Big(   N(T) > n^{2k+2}  \Big). 
\end{align*}
Here $\sigma_0$ is a fixed $(k-2)$-simplex.   Clearly, 
\begin{align*}
 \P\big(   N(T) > n^{2k+2}  \big) &\le \frac{\E [N(T)]}{n^{2k+2}} =\mathcal O (n^{-k-1});
\end{align*}
so, by  Lemma \ref{l:pureness} and the stationarity of $X_k$, 
\begin{align*}
&\P \big( \sup_{0\le t \le T} \beta_{k-1}(t) =0 \big) \\
&\ge 1-n^{3k+1}
    \P\Big(     \lambda_2 \big(  \text{lk}_{X_k(\eta_\ell)}(\sigma_0) \big) \le 1-\frac{1}{k}  \text{ or } \text{lk}_{X_k(\eta_\ell)}(\sigma_0) \text{ is disconnected}  \Big)  - \mathcal O(n^{-k-1}).
\end{align*}
Given $\sigma_0\in X_k$, we have by Lemma  \ref{l:log.L.sigma} and its proof that, for some $\epsilon>0$, 
\begin{align*}
&\P \Big(   \lambda_2 \big(  \text{lk}_{X_k(\eta_\ell)}(\sigma_0) \big) \le 1-\frac{1}{k}  \text{ or } \text{lk}_{X_k(\eta_\ell)}(\sigma_0) \text{ is disconnected}  \Big) \\
&= \P \Big(  \Big\{   \lambda_2 \big(  \text{lk}_{X_k(\eta_\ell)}(\sigma_0) \big) \le 1-\frac{1}{k}  \text{ or } \text{lk}_{X_k(\eta_\ell)}(\sigma_0) \text{ is disconnected}  \Big\} \\
&\qquad \qquad \cap \Big\{ \frac{(1+\delta)\log L_{\sigma_0}}{L_{\sigma_0}} \le p_1, \ L_{\sigma_0} \ge \frac{n^{1-\psi_{k-1}(\balpha)}}{2} \Big\} \Big)+ o(e^{-n^\epsilon}) \\
&=\sum_{m=1}^{n-k+1} \P \Big(   \lambda_2 \big(  \text{lk}_{X_k(\eta_\ell)}(\sigma_0) \big) \le 1-\frac{1}{k}  \text{ or } \text{lk}_{X_k(\eta_\ell)}(\sigma_0) \text{ is disconnected}  \, \Big|\, L_{\sigma_0}=m \Big) \\
&\qquad \qquad \qquad \qquad \times \one \Big\{  \frac{(1+\delta)\log m}{m} \le p_1, \ m \ge \frac{n^{1-\psi_{k-1}(\balpha)}}{2}  \Big\}\, \P(L_\sigma=m) + o(e^{-n^\epsilon}). 
\end{align*}
However, $\text{lk}_{X_k}(\sigma_0) | L_{\sigma_0}=m$ has the law of
the Erd\"os-R\'enyi graph with parameters $m$ and $p_1$; see Lemma~4.2
in \cite{fowler:2019}. Furthermore, in the range of $m$ we are
considering, $p_1 \ge \frac{(1+\delta)\log m}{m}$. It follows from 
the spectral gap theorem of Theorem 1.1 in
\cite{hoffman:kahle:paquette:2019} that for some $\delta$-dependent constant $C$, 
$$
\P \Big(    \lambda_2 \big(  \text{lk}_{X_k(\eta_\ell)}(\sigma_0) \big) \le 1-\frac{1}{k}  \text{ or } \text{lk}_{X_k(\eta_\ell)}(\sigma_0) \text{ is disconnected} \, \Big|\, L_\sigma=m \Big) \leq Cm^{-\delta}. 
$$
We conclude that 
\begin{align*}
&\P \Big(    \lambda_2 \big(  \text{lk}_{X_k(\eta_\ell)}(\sigma_0) \big) \le 1-\frac{1}{k}  \text{ or } \text{lk}_{X_k(\eta_\ell)}(\sigma_0) \text{ is disconnected} \Big) \\
&\quad \le C\Big(
                 \frac{n^{1-\psi_{k-1}(\balpha)}}{2} \Big)^{-\delta}
                 + o(e^{-n^\epsilon}) = \mathcal
                 O\big(n^{-\delta(1-\psi_{k-1}(\balpha))} \big),  
\end{align*}
and so 
$$
\P\Big(\sup_{0\le t \le T}\beta_{k-1}(t)= 0 \Big)  \ge 1 - \mathcal O\big(n^{3k+1-\delta(1-\psi_{k-1}(\balpha))} \big) - \mathcal O(n^{-k-1}). 
$$
As $1-\psi_{k-1}(\balpha)>0$, the claim follows by taking large enough $\delta>0$. 
\end{proof}
We can now complete the proof of the proposition. Since
$\Var(f_k)\to\infty$,  we have  for any
$0<T< \infty$ and $\epsilon>0$, using Lemma  \ref{l:betti.k-1}, 
\begin{align*}
&\P \Big( \sup_{0\le t \le T} \big|  \beta_{k-1}(t)-\E(\beta_{k-1}) \big| > \epsilon \sqrt{\Var(f_k)} \Big) \le \frac{2}{\epsilon} \E \Big( \sup_{0\le t \le T} \beta_{k-1}(t) \Big) \\
&\le \frac{2}{\epsilon} \E \Big( \sup_{0\le t \le T} f_{k-1}(t)\, \one
    \big\{ \sup_{0\le t \le T} \beta_{k-1}(t) \neq 0 \big\} \Big)  \le \frac{2}{\epsilon} \binom{n}{k} \P\Big( \sup_{0\le t \le T} \beta_{k-1}(t) \neq 0 \Big)\\
&=\frac{2}{\epsilon} \mathcal O(n^{-1}) \to 0, \ \ \ n\to\infty,
\end{align*}
as required. 
\end{proof}

\subsection{Representation of the Betti number}
In this section we verify \eqref{e:beta.k+1}. 
\begin{proposition}  \label{p:betti.representation}
For $\ell \ge 1$, 
\begin{equation}  \label{e:assertion.representation}
\beta_\ell(t) = \beta_\ell \big( X([n],\bp; t) \big) = \sum_{j=\ell+2}^n \sum_{r \ge 1} \sum_{\sigma \subset [n], \, |\sigma|=j} r\eta_\sigma^{(j,r,\ell)}(t), 
\end{equation}
where $\eta_\sigma^{(j,r,\ell)}(t)$ is the indicator function in \eqref{e:beta.k+1}. 
\end{proposition}
\begin{proof}
For $\ell$-simplices $\sigma, \tau$ in $X([n],\bp; t)$, write $\sigma \sim
\tau$ if  they can be connected by a sequence of $\ell$-simplices
$\sigma=\sigma_0, \sigma_1, \dots, \sigma_{j-1}, \sigma_j =\tau$ such
that dim$(\sigma_i \cap \sigma_{i+1})=\ell-1$, $0 \le i \le
j-1$. Clearly $\sim$ is an equivalence relation. Consider the
equivalence classes $\G_1, \dots, \G_N$  associated with this
relation. For each $i=1,\dots,N$, let $X_i$ be the smallest subcomplex
of $X([n],\bp; t)$ containing all the simplices for which some
$\ell$-simplex in $\G_i$ is a face.  Then $X_i$
is necessarily a maximal $\ell$-strongly connected subcomplex, such
that dim$(X_{i_1}\cap X_{i_2}) \le \ell-2$ for any distinct $1 \le i_1
\neq i_2 \le N$.  
Let  $X^{(N)} := \bigcup_{i=1}^NX_i$ and let $X_{N+1}$ be a subcomplex
of $X([n],\bp; t)$ containing all simplices in $X([n],\bp;
t)\setminus X^{(N)}$. By construction dim$(X_{N+1}) \le \ell-1$ and
dim$(X_{N+1}\cap X^{(N)}) \le \ell-2$. 
With
this setup, establishing the claim of the proposition reduces to
proving the following statements: 
\begin{equation}\label{e:first.representation}
\beta_\ell \big( X([n],\bp;t)  \big)=  \beta_\ell(X^{(N)}), 
\end{equation}
and 
\begin{equation}\label{e:second.representation}
\beta_\ell(X^{(N)}) = \sum_{i=1}^N\beta_\ell(X_i). 
\end{equation}
Indeed, since $\sum_{i=1}^N\beta_\ell(X_i)$ in \eqref{e:second.representation} is clearly equal to the right hand side of \eqref{e:assertion.representation}, our proof will be done once \eqref{e:first.representation} and \eqref{e:second.representation} are both established. 
For the proof of \eqref{e:first.representation} we exploit the
following Mayer-Vietoris exact sequence: 
\begin{align*}
\dots \rightarrow  H_\ell\big( X^{(N)}\cap X_{N+1} \big) &\stackrel{\lambda_\ell}{\to} H_\ell\big( X^{(N)} \big) \oplus H_\ell(X_{N+1}) \to H_\ell \big( X([n],\bp;t) \big) \\
&\to H_{\ell-1}\big( X^{(N)}\cap X_{N+1} \big) \stackrel{\lambda_{\ell-1}}{\to} H_{\ell-1} \big( X^{(N)} \big) \oplus  H_{\ell-1}(X_{N+1}) \to \dots,
\end{align*}
where $H_\ell$ represents the homology group of order $\ell$, and
$\lambda_\ell = (\lambda^{(1)}_\ell, \lambda^{(2)}_\ell)$ denotes the
homomorphism induced by the inclusions $X^{(N)} \cap X_{N+1}
\hookrightarrow X^{(N)}$ and $X^{(N)} \cap X_{N+1} \hookrightarrow
X_{N+1}$. An elementary rank calculation (see e.g., Lemma 2.3 in \cite{yogeshwaran:subag:adler:2017}) yields 
\begin{align*}
\beta_\ell \big( X([n],\bp;t) \big) = \beta_\ell \big( X^{(N)} \big) +
  \beta_\ell (X_{N+1}) + \text{rank} (\text{ker} \lambda_\ell) +
  \text{rank} (\text{ker} \lambda_{\ell-1}) - \beta_\ell
  \big(X^{(N)}\cap X_{N+1} \big).  
\end{align*}
Since dim$(X_{N+1})\le \ell -1$ and dim$\big( X^{(N)}\cap X_{N+1}\big) \le \ell-2$, we have that 
$$
H_\ell (X_{N+1}) \cong 0, \ \ H_\ell \big( X^{(N)}\cap X_{N+1} \big)
\cong 0, \ \ \ H_{\ell-1}\big( X^{(N)}\cap X_{N+1} \big) \cong 0. 
$$
 In particular, ker$\lambda_\ell$ and ker$\lambda_{\ell-1}$ are both trivial. Combining all these observations we obtain \eqref{e:first.representation}. 

We now turn to deriving \eqref{e:second.representation}. The statement
is trivial for $N=1$. If $N>1$, we denote $X^{(j)} :=
\bigcup_{i=1}^jX_i$ and prove that $\beta_\ell\big( X^{(j)} \big) =
\sum_{i=1}^{j}\beta_\ell (X_i)$ for $j=1,\ldots, N$ inductively. 
Once again, the case   $j=1$ is trivial, so suppose for induction that
$\beta_\ell\big( X^{(j-1)} \big) = \sum_{i=1}^{j-1}\beta_\ell
(X_i)$ for some $1\leq j<N$. We consider another Mayer-Vietoris exact
sequence, given by   
\begin{align*}
\dots \rightarrow  H_\ell\big( X^{(j-1)}\cap X_j \big) &\stackrel{\nu_\ell}{\to} H_\ell\big( X^{(j-1)} \big) \oplus H_\ell(X_j) \to H_\ell \big( X^{(j)} \big) \\
&\to H_{\ell-1}\big( X^{(j-1)}\cap X_j \big) \stackrel{\nu_{\ell-1}}{\to} H_{\ell-1} \big( X^{(j-1)} \big) \oplus  H_{\ell-1}(X_j) \to \dots,
\end{align*}
where $\nu_\ell$, $\nu_{\ell-1}$ are group homomorphisms analogous to
the earlier situation. Since dim$\big( X^{(j-1)}\cap X_j \big)\le
\ell-2$, the same rank computation as above gives us 
$$
\beta_\ell \big( X^{(j)} \big) = \beta_\ell \big( X^{(j-1)} \big) + \beta_\ell (X_j) + \text{rank} (\text{ker} \nu_\ell) + \text{rank} (\text{ker} \nu_{\ell-1}) - \beta_\ell \big(X^{(j-1)}\cap X_j \big) = \sum_{i=1}^j \beta_\ell(X_i),
$$
completing the induction step. 
\end{proof}

\noindent \textbf{Acknowledgement}: The authors would like to thank
the anonymous referee and the Associate Editor for their comments that
lead to a substantial improvement of the paper.



\end{document}